%% file: main_plain_arxiv.tex
\documentclass{article}

\usepackage{graphicx,amsmath,amsfonts,amssymb,amsthm}

\usepackage{cleveref}
\usepackage{color}
\usepackage{url}

\newtheorem{theorem}{Theorem}
\newtheorem{lemma}[theorem]{Lemma}
\newtheorem{proposition}[theorem]{Proposition}
\newtheorem{corollary}[theorem]{Corollary}
\newtheorem{remark}[theorem]{Remark}
\newtheorem{property}[theorem]{Property}

\numberwithin{theorem}{section}

\numberwithin{equation}{section}

\crefname{property}{Property}{Properties}
\crefname{lemma}{Lemma}{Lemmas}
\crefname{theorem}{Theorem}{Theorems}
\crefname{proposition}{Proposition}{Propositions}
\crefname{corollary}{Corollary}{Corollaries}
\crefname{equation}{}{}

\title{Integral equations for flexural-gravity waves: analysis and numerical methods}

\author{T. Askham\thanks{Department of Mathematical Sciences, New Jersey Institute of Technology, Newark, NJ 07102 (\texttt{askham@njit.edu}).} \and J.G. Hoskins\thanks{Department of Statistics and CCAM, University of Chicago, Chicago, IL 60637 (\texttt{jeremyhoskins@uchicago.edu}).} \and P. Nekrasov\thanks{Committee on Computational and Applied Mathematics, University of Chicago, Chicago, IL 60637 (\texttt{pn3@uchicago.edu}).} \and M. Rachh\thanks{Center for Computational Mathematics, Flatiron Institute, New York, NY 10010 (\texttt{mrachh@flatironinstitute.org}).}}

\newcommand{\bbR}{\mathbb{R}}

\newcommand{\bK}{\mathbf{K}}

\DeclareMathOperator*{\Res}{Res}

\newcommand{\bbC}{\mathbb{C}}
\newcommand{\dd}{\textrm{d}}
\newcommand{\GS}{G_{\mathrm{S}}}
\newcommand{\Gphi}{G_{\phi}}
\newcommand{\brthree}{\mathbf{r}_{\textrm{3d}}}
\newcommand{\brthreet}{\mathbf{r}_{\textup{3d}}}
\newcommand{\br}{\mathbf{r}}
\newcommand{\bx}{\mathbf{x}}
\newcommand{\by}{\mathbf{y}}
\newcommand{\bxi}{\mathbf{x}_{\bf i}^h}
\newcommand{\bxip}{\mathbf{x}_{{\bf i}'}^h}

\newcommand{\bxl}{\mathbf{x}_{\boldsymbol{\ell}}^h}
\newcommand{\bell}{\boldsymbol{\ell}}
\newcommand{\bxlp}{\mathbf{x}_{\boldsymbol{\ell}'}^h}
\newcommand{\inc}{\textrm{inc}}
\newcommand{\scat}{\textrm{scat}}
\newcommand{\erf}{\textrm{erf}}
\newcommand{\alphn}{{\alpha_0}}
\newcommand{\betan}{{\beta_0}}
\newcommand{\threed}{\textrm{3d}}
\newcommand{\threedt}{\textup{3d}}
\newcommand{\pv}{\operatorname{p.\!v.}}
\newcommand{\hide}[1]{}

\begin{document}

\maketitle

\begin{abstract}
In this work, we develop a fast and accurate method for the scattering of flexural-gravity waves by a thin plate of varying thickness overlying a fluid of infinite depth. This problem commonly arises in the study of sea ice and ice shelves, which can have complicated heterogeneities that include ridges and rolls. With certain natural assumptions on the thickness, we present an integral equation formulation for solving this class of problems and analyze its mathematical properties. The integral equation is then discretized and solved using a high-order-accurate, FFT-accelerated algorithm. The speed, accuracy, and scalability of this approach are demonstrated through a variety of illustrative examples.
\end{abstract}

\section{Introduction}

The motion of a thin plate coupled to a fluid is a subject of longstanding interest in the fields of geophysics and acoustics. In the polar regions, ice sheets floating atop the ocean are often modeled as thin plates that bend and flex in response to hydrodynamic forces. The coupling of plates to acoustic media results in similar wave phenomena and has important implications for aerodynamic stability \cite{miles1956, graham1996boundary} and sound radiation \cite{gorman1991plate, laulagnet1998sound}. In complicated media, these waves exhibit dynamics that include branching \cite{jose2022branched, jose2023branched}, bending \cite{andronov1997passage, darabi2018broadband}, and cloaking \cite{li2015acoustic, zareei2017broadband}. 

In the Arctic, ocean waves interact with sea ice in a dynamic region known as the marginal ice zone. Here, short period waves are attenuated by relatively thin ice, allowing longer periods to travel deeper into the ice pack \cite{wadhams1988}. The same is true in the Southern Ocean, where the opening of sea ice-free corridors allows the sea swell to make contact with Antarctic ice shelves \cite{teder2022sea, cathles2009seismic}. Meanwhile, tsunami and infragravity waves, largely unattenuated by sea ice, are able to reach the ice shelf to induce significant flexure and calving \cite{bromirski2017tsunami, brunt2011antarctic}. 
While it is difficult to separate the impact of ocean wave forcing from thermodynamic processes, there is growing evidence to support the role of ocean waves in the breakup of Antarctic ice shelves and  Arctic sea ice \cite{asplin2012, kohout2014storm, kohout2016situ,walker2024multi}. 
To weigh these factors, it is critical to understand and model interactions between ice and the ocean at various scales.  The problem is inherently multi-scale, since sharp ice features like ridges account for a large portion of the total ice area \cite{hibler1972statistical}, and can significantly affect wave attenuation and propagation in the polar regions.

In the small thickness limit, the ice-cover can be treated as the boundary of the fluid domain wherein the velocity potential satisfies a fourth order differential equation whose coefficients vary with the thickness of the ice-cover. Extensive work has been done to model the propagation of flexural-gravity waves across various obstructions in ice-covers of uniform thickness. These include narrow cracks \cite{squire2000analytic, porter2006scattering}, wide cracks \cite{chung2005reflection}, frozen leads \cite{barrett},  and icebergs \cite{squire2001modelling}; see \cite{squire2007ocean} for a comprehensive overview of these works. In this regime, semi-analytic solutions can be obtained using eigenfunction expansions \cite{fox1994, sahoo2001scattering, evans2003wave}, Carlemann singular integral equations \cite{chakrabarti2000solution}, or the Wiener-Hopf method \cite{kouzov1963diffraction, chung2002propagation, chung2002calculation, balmforth1999ocean}. Interestingly, many of these works find that ice features act as low-pass filters, which may help to explain the attenuation patterns observed in the Arctic.

Less attention has been given to the problem of wave scattering induced by continuous changes in the thickness of the ice, likely because they are more resistant to analytical techniques. Early attempts to solve this problem assume that ice features can be modeled as point inhomogeneities \cite{andronov1990acoustic, marchenko1995natural} or rely on the shallow-water approximation \cite{marchenko1997surface}. In three dimensions, this problem has been approximated using a \emph{ray method} \cite{hermans2003ray}, or by using a truncated modal decomposition of the vertical component of the velocity \cite{porter2004approximations}; though for the latter, numerical procedures have only been implemented in two dimensions \cite{bennetts2007multi} or for axisymmetric ice floes \cite{bennetts2009wave}. An alternate approach is to model the full-thickness of ice which has been used to study wave propagation in crevasses, see~\cite{sergienko17,ice2, nekrasov2023rolls}, for example. However, these continuum models  become prohibitively expensive in three dimensions, motivating the design of fast methods that work in the thin-plate regime. Perhaps the most closely related prior method to the present work is~\cite{williams2004oblique}. Using Green's identities, the authors construct integral equations for unknowns supported on the water-ice interface to approximate the reflected and transmitted coefficients for incident plane waves in two dimensions.

In this work, we derive a novel integral equation formulation for the fully three-dimensional scattering problem, assuming the thickness is smoothly varying and the region of inhomogeneity is compactly supported. We do this by reducing the problem to an associated boundary integral equation on the ice-covered surface. The reduction occurs in two steps. First, we use standard potential theory techniques to convert the three-dimensional boundary value problem into an integro-differential equation for an unknown defined on the surface. In the second step, we further reduce this integro-differential equation to an integral equation defined only on the compact support of the perturbation of the thickness. The resulting integral equation is an {\hide{red}integral equation of the second kind} and, for a large range of parameters, is provably well-posed. Conveniently, existing numerical methods can
be easily adapted to discretize this integral equation with high order accuracy and it is straightforward to establish
the convergence of the resulting scheme. Further, the discretized system is amenable to solution by fast algorithms,
enabling the scalability of this approach to large scattering problems and (essentially) arbitrary accuracy. 

In the solution of the integro-differential equation resulting from the first step of the reduction, one obtains a slowly decaying surface density, a problem that often arises with infinite interfaces. An alternate strategy to the one presented here is the technique of PML-BIE \cite{pmlbie}, or coordinate complexification \cite{dir_comp}, which involves complexifying the boundary so that the density decays exponentially and the integration can be performed over a finite region. This method was recently applied in \cite{bonnetbendhia} for the solution of two-dimensional water waves with one-dimensional boundary, and we expect that the same can be readily extended to three-dimensions. For this approach, even though the coordinates are complex, high order quadrature schemes~\cite{acha} and fast algorithms are available. In the present context, the approach adopted here has the advantage of being substantially simpler both to implement and analyze, as well as more amenable to acceleration via traditional fast algorithms. More broadly, integral equation methods of this flavor have been widely applied to scattering problems, and while we do not 
seek to review them here, we remark that these methods have appeared in several related contexts, including 
capillary surfers \cite{de2018capillary, oza2023theoretical},
 obstacles floating amid surface waves \cite{bonnetbendhia}, and surface waves on interfaces between insulating materials \cite{cpam2025,dirac2024}.

The remainder of this paper is as follows. In \Cref{sec:problem}, we outline the boundary value problem for flexural-gravity waves, and features of the model. In \Cref{sec:inteq}, we describe the reduction of the PDE to its associated surface integral equation. Following this, in \Cref{sec:analyticalie} we state key analytical properties of the operators appearing in this integral equation, and in \Cref{sec:existenceuniqueness} we prove existence and uniqueness results. The numerical procedure for solving this integral equation and its convergence are briefly described in \Cref{sec:numerics}, and we present a number of applications of these methods in \Cref{sec:examples}. Finally we discuss the limitations of the present work and areas for future research in \Cref{sec:conclusion}.

\section{Problem}\label{sec:problem}

\begin{figure}[ht]
  \centering
  \includegraphics[width=4in]{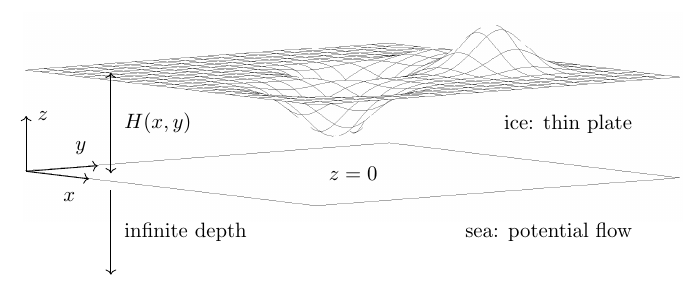}
  \caption{Illustration of the problem setup.}
  \label{fig:setup}
\end{figure}

Let the fluid domain be defined by the lower half-space $\{ (x,y,z)  :  z \leq 0 \}$, and let the ice-covered surface be its top boundary. Then, we have the following boundary value problem:
\begin{align}
  \Delta \phi + \partial_z^2 \phi  = 0 \, , &  & z < 0 \, ,   \label{eq:phipde1}
 \\
   ( \alpha \Delta^2_{\rm S} - \beta) \partial_z \phi 
  + \gamma \phi = 0 \, , &   & z = 0 \, ,  \label{eq:phipde2} 
\end{align}
supplemented with suitable decay conditions defined in \Cref{sec:existenceuniqueness}. Here, $\phi$ is the velocity potential, 
$\Delta := \frac{\partial^2}{\partial x^2} +  \frac{\partial^2}{\partial y^2} $ is the surface Laplacian, and $\Delta^2_{\rm S}$ is the modified biharmonic operator defined below. Though equations like these arise in a number of applications related to acoustics, the setting which motivates the current work is the dynamics of ice-covered oceans. In this context, the coefficients have the following physical connotations: $\alpha  $ is the flexural rigidity of the ice, $\beta  $ is the difference in the ice's inertia and gravitational force acting on the fluid, and $\gamma$ is related to the hydrodynamic pressure: 
$$\alpha := \frac{EH^3}{12(1-\nu^2)} \, , \quad \beta := \rho_{\textrm{ice}} H \omega^2 - \rho_{\textrm{sea}} g \, , \quad \gamma := - \rho_{ \textrm{sea}} \omega^2 \, , $$
where $E$ is the Young's modulus, $\nu$ is the Poisson ratio, $H$ is the thickness, $\rho_{ \textrm{ice}}$ is the density of ice, $\rho_{ \textrm{sea}}$ is the density of ocean water, $g$ is acceleration due to gravity, and $\omega$ is the angular frequency of the hydroelastic wave, which is a fixed parameter of the problem. See  \Cref{fig:setup} for an illustration of the problem setup and \Cref{table:1} for common values of these physical parameters. When the flexural rigidity $\alpha$ is allowed to vary in space, the modified biharmonic operator $\Delta_{\rm S}$ acting on the surface $z = 0$ is given by 
\begin{align*}
    \Delta^2_{\rm S}:= \frac{1}{\alpha} \Delta( \alpha \Delta) + \frac{1 - \nu}{\alpha} \left( 2 \frac{\partial^2 \alpha}{\partial x \partial y } \frac{\partial^2 }{\partial x \partial y } - \frac{\partial^2 \alpha}{\partial x^2} \frac{\partial^2 }{\partial y^2}  - \frac{\partial^2 \alpha}{\partial y^2} \frac{\partial^2 }{\partial x^2}  \right) \, .
\end{align*}  
This operator appears in many related works, including \cite{porter2004approximations,bennetts2007multi,Reissner, PAPATHANASIOU2019102741}. We note that when the flexural rigidity $\alpha$ is constant, this operator simplifies to $\Delta_{\rm S}^2 =  \Delta^2$.

In the present study, we consider the case where the material properties of the ice are allowed to vary in space, i.e. $\alpha(x,y) = \alpha_0 + \alpha_c(x,y) $, 
$\beta(x,y) = \beta_0 + \beta_c(x,y)$, where {\hide{red}$\alpha_0 \in\mathbb{R}_+$, $\beta_0 \in \bbC$, and $\gamma \in \mathbb{R}_-$ 
are constants}, and $\alpha_c$ and $\beta_c$ are compactly supported on a bounded region $\Omega \subset \mathbb{R}^2,$ {\hide{red}with $\alpha_c$ real-valued and satisfying $\alpha_0+\alpha_c>0$ and $\Im (\beta_c) \ge 0$}. Note that the most common source of variation in these coefficients is changes in the thickness $H$; however, the methods and analysis described here are sufficiently general to encompass contexts in which other material properties such as density or Young's modulus also vary. If $\Im{(\beta_0)} > 0$ then we say that the equations are in the dissipative regime, while if $\Im{(\beta_0)} = 0$ we say that they are in the non-dissipative regime. 

The model given by \cref{eq:phipde1,eq:phipde2} assumes infinite depth, which is appropriate when the fluid depth significantly exceeds the  characteristic wavelength, as is the case for Arctic sea ice which sits above deep water where there is minimal interaction with the ocean floor \cite{christakos2024long}. This model also assumes the well-known shallow draft approximation \cite{buchner1993evaluation}, which allows one to treat the undersurface of the ice as the fixed surface $z=0$ and assume that all thickness changes manifest in the upper surface of the ice (see  \Cref{fig:setup}), even though such configurations may not be hydrostatically equilibrated. 
Previous studies that have looked at the effect of draft \cite{porter2004approximations, vaughan2006scattering} suggest that the governing equations depend merely on the distance between the keel and the sea bottom, so we do not expect the dynamics to be significantly altered in the infinite depth setting. Both the plate and fluid displacements are assumed to be small, allowing one to neglect quadratic terms in the equations of motion and fix the boundary.

\begin{table}
\centering
\begin{tabular}{c|c|c|c}
     Description & Parameter & Values & Reference  \\
    \hline\hline
  Young's modulus & $E$ & $7\cdot10^9$  kg m$^{-1}$ s$^{-2}$ & \cite{weeks1967mechanical} \\
 \hline Poisson ratio &  $\nu$ & 0.33 & \cite{weeks1967mechanical}  \\
 \hline Ice thickness & $H$ & 0.1 m -- 1 km & \cite{landy2022year, hogg2021extending}  \\
 \hline Density of seawater & $\rho_{ \textrm{sea}}$ & 1025  kg m$^{-3}$ &   \\
 \hline Density of ice & $\rho_{\textrm{ice}}$  & 917 kg m$^{-3}$ & \\
\hline Gravitational acceleration & $g$ & 9.8 m s$^{-2}$ \\
 \hline Angular frequency & $\omega$ &  0.02 -- 60 rad s$^{-1}$ & \cite{bromirski2012response, balmforth1999ocean}  \vspace{0.3cm} 
\end{tabular}
\caption{Physical parameters of the model and their approximate values. The thickness of the ice can range from ordinary sea ice (0.1 - 4 m), to sea ice  ridges (5 - 30 m), to Antarctic ice shelves (hundreds of meters to 1 km). The frequency of wave forcing can also vary considerably -- from low-frequency tsunami and infragravity waves (0.004 - 0.01 Hz),  to ordinary gravity waves (0.05 - 1 Hz) and higher resonances.}
\label{table:1}
\end{table}

\section{Reduction to an integral equation}\label{sec:inteq}
Suppose that a harmonic incident field $\phi^\inc$ is prescribed and
that $\phi = \phi^\inc + \phi^\scat$. The scattered field,
$\phi^\scat$ then satisfies  
\begin{equation}
   (\alpha \Delta_{{\mathrm S}}^2 - \beta) \partial_z \phi^\scat
  + \gamma \phi^\scat = f \,  , \quad z = 0 \, , \label{eq:phibc1}
\end{equation}
where
$$ f = -(\alpha \Delta_{{\mathrm S}}^2 - \beta) \partial_z \phi^\inc
- \gamma \phi^\inc \; .$$

We represent the scattered velocity potential $\phi^\scat$ by a
single layer potential with a charge density, $\sigma$,
defined on the surface $z=0$, i.e. 
\begin{equation}
  \label{eq:Sdef}
 \phi^\scat(\brthree) =  S[\sigma](\brthree) := \pv \int_{\bbR^2}
\frac{1}{ 4\pi \left |\brthree-\brthree' \right|} \sigma(\br') \,
\dd A(\br') \; ,
\end{equation}
where the principal value ($\pv$) integral is defined as
\begin{equation}
\label{eq:pvdef} \pv \int_{\bbR^2} g(\br') \, \dd A(\br') =
\lim_{R\to \infty} \int_{|\br'|\leq R} g(\br') \, \dd A(\br') 
\end{equation}
and 
$$ \brthree = \begin{pmatrix} x \\ y \\ z
\end{pmatrix} \; , \quad \brthree' = \begin{pmatrix} x' \\ y' \\ 0 
\end{pmatrix} \; ,   \quad \textrm{and} \quad  \br' = \begin{pmatrix} x' \\ y'
\end{pmatrix} \; .$$
Note that the $\pv$ integral is only necessary when the integrand is
not absolutely integrable. 
The potential is {\hide{red}well-defined provided that $\sigma$ satisfies certain decay and/or outgoing conditions at infinity (see \cref{lem:pvsinglelayer} for sufficient conditions).} This reduces the problem to determining
$\sigma$ such that boundary condition
\cref{eq:phibc1} is satisfied. 

Substituting the single layer representation of $\phi^\scat$
into \cref{eq:phibc1} and applying standard integral
``jump formulas'' in the limit $z\to 0^-$
yields
\begin{equation}
  \label{eq:integrodiffsigma}
  \frac{1}{2} \left(\alpha \Delta^2_{{\mathrm S}} - \beta  \right) \sigma(\br)
+ \gamma \pv \int_{\bbR^2} \frac{1}{4\pi |\br - \br'|} \sigma(\br') \,
\dd A(\br')  = f(\br) \; ,
\end{equation}
for each $\br \in \bbR^2$. This is an integro-differential equation for $\sigma$. Note that a similar integro-differential equation in terms of plate displacement appears in \cite{hermans2003ray} with some flexural terms omitted.

The remainder of this section derives an integral
equation formulation for the surface problem~\cref{eq:integrodiffsigma}.
\Cref{sec:constantcoeff} describes an appropriate
Green's function for the constant coefficient case.
\Cref{sec:Lippmann} derives a variable coefficient
integral equation for $\sigma$ based on the (adjoint) Lippmann-Schwinger
formalism and the Green's function of~{\Cref{sec:constantcoeff}.

\subsection{The constant coefficient problem}
\label{sec:constantcoeff}

When the rigidity is constant, we have that $\Delta_{{\mathrm S}}^2 =
\Delta^2$. We consider the fundamental solution, $\GS$, of the
constant coefficient version of \cref{eq:integrodiffsigma}, i.e. 
\begin{equation}
  \label{eq:GSid}
  \frac12 \left( \alphn \Delta^2 - \betan  \right) \GS(\br, \br')
+ \gamma \pv \int_{\bbR^2} \frac{1}{4\pi \left|\br - \br'' \right|} \GS(\br', \br'') \,
\dd A(\br'')  = \delta(\br,\br') \; ,
\end{equation}
where $\alphn$, $\betan$, and $\gamma$ are constants.
We are also interested in the on-surface ($z=0$) value of the single layer
potential applied to this fundamental solution,
\begin{equation}
  \label{eq:Gphiid}
\Gphi (\br,\br') :=  \pv \int_{\bbR^2}
\frac{1}{ 4\pi \left |\br-\br'' \right|} \GS(\br'',\br') \,
\dd A(\br'') \; ,
\end{equation}
which is the velocity potential corresponding to an impulse
response.

The appropriate (outward radiating) solution of \cref{eq:GSid}
can be derived via the limiting absorption principle. In particular,
we first derive the solution of \cref{eq:GSid} for dissipative plates and take the limit
from the complex upper half plane to obtain
the appropriate radiating solution for real-valued $\betan$.

Before stating the formulas, we require the following lemma that establishes
some basic facts and notation
related to the Fourier multiplier on the left hand side of~\cref{eq:GSid}. 
Its proof is straightforward and omitted.
\begin{lemma}
  \label{lem:dispersionpoly}
  Let $\rho_1,\ldots,\rho_5$ denote the
  roots of the polynomial
  \begin{equation}
    \label{eq:dispersion}
     p(z) = \alpha_0 z^5 - \beta_0 z + \gamma \; .
  \end{equation}
  If $\Im(\betan)\ne 0$, then none of the $\rho_j$ are real. If
  $\betan \in \bbR$, then exactly one of the $\rho_j$ is a positive real number,
  which we label by $\rho_1$.

Supposing that all of the roots are distinct, 
we define the coefficients $e_1,\ldots,e_5$ by\begin{equation}  \label{eq:moments}
    e_j = \Res_{z=\rho_j} \frac{1}{p(z)} \quad \textrm{for each}
  \quad j=1,\ldots,5 \; . 
  \end{equation}
  Additionally, the coefficients $e_j$ and roots $\rho_j$ satisfy the following relations
  \begin{equation}
    \sum_{j=1}^5 e_j \rho_j^q = 0 \textrm{ for } q\in\{0,1,2,3,5,6,7 \} \, , \
    \sum_{j=1}^5 e_j \rho_j^4 = \frac{1}{\alphn} \, . \label{eq:momentrelation}
  \end{equation}
  
\end{lemma}
{\hide{red}The previous lemma combined with the analytic implicit function theorem immediately implies the following corollary.
\begin{corollary}\label{cor:analyticity_of_root}
For $\Im (\beta_0)=0,$ let $\rho(\beta_0)$ denote the unique simple positive root of $p(z).$ Then $\rho$ extends to an analytic function in a neighborhood of each $\beta_\ast \in \mathbb{R},$ with $\operatorname{sgn} (\Im ( \rho (\beta_0))) =\operatorname{sgn} (\Im (\beta_0)),$ and $e_1(\beta_0) = 1/p'(\rho(\beta_0))$ analytic and nonvanishing.    
\end{corollary}}
\begin{remark}
  The polynomial \cref{eq:dispersion} corresponds to the deep water dispersion
  relation for flexural-gravity waves \cite{fox1994,meylan2018dispersion}. 
  The moment relations \cref{eq:momentrelation}
  also appear in \cite{squire2001region, squire2001modelling}.
  The locations of the roots of \cref{eq:dispersion} dictate the
  behavior of $\GS$ at infinity. For ease of exposition, the
  formulas will be stated for the case that the roots
  $\rho_1,\ldots,\rho_5$ are distinct. {\hide{red} If $\beta_0$ is sufficiently large
  and positive (corresponding to large $H$), there are two negative real roots and a 
  double root occurs when the behavior transitions from one to three
  real roots.} The results 
  can easily be extended to {\hide{red} the case of a double root}
  (the form of the polynomial $p(z)$ precludes three or more roots
  coinciding); details are given in~\cref{app:greensfun}.
\end{remark}

The following theorem provides an explicit formula for the fundamental solution
of \cref{eq:GSid} and the corresponding velocity potential
\cref{eq:Gphiid}. Note that $\GS$ also appears in a slightly different form in~\cite{fox1999green}, and
similar Green's functions appear in problems involving capillary
surfers \cite{de2018capillary, oza2023theoretical}.
\begin{theorem}
  \label{thm:greensfun}
  Let the roots $\rho_1,\ldots,\rho_5$ and the coefficients
  $e_1,\ldots,e_5$ be as in the statement of \cref{lem:dispersionpoly}
  and suppose that the roots are distinct.

  If $\Im(\betan) \ne 0$, then
  \begin{equation}
    \label{eq:GSeps}
    \GS(\br,\br') = \frac{1}{2} \sum_{j=1}^5 e_j \rho_j^2
    \bK_0(-\rho_j |\br - \br'|) \; ,
  \end{equation}
  where ${\bf K}_0$ is the Struve function in standard notation~\cite{NIST}.  
  For real-valued $\betan$, consider the solution of \cref{eq:GSid} 
  with $\betan$ replaced by  
  $\beta_\epsilon = \betan + i\epsilon$. The limit of $\GS$ as $\epsilon\to 0^+$
  is 
  \begin{multline}
    \label{eq:GS}
    \GS(\br,\br') = \frac{1}{2} e_1 \rho_1^{2} \left  [
      - \mathbf{K}_0(\rho_1 |\br - \br'|) + 2 i H_0^{(1)} (\rho_1 |\br - \br'|)
      \right ]  \\ + \frac{1}{2} \sum_{j=2}^5 e_j \rho_j^2
    \bK_0(-\rho_j |\br - \br'|) \; ,
  \end{multline}
  where $H^{(1)}_0$ is the Hankel function of the first kind.
  
  The value of the
  corresponding velocity potential, given by \cref{eq:Gphiid}, is 
  \begin{equation}
    \label{eq:Gphieps}
    \Gphi(\br,\br') = \frac{1}{4} \sum_{j=1}^5 e_j \rho_j
    \bK_0(-\rho_j |\br - \br'|) \; ,
  \end{equation}
  in the case that $\Im(\betan)\ne 0$ and 
  \begin{multline}
    \label{eq:Gphi}
    \Gphi(\br,\br') = \frac{1}{4} e_1 \rho_1 \left  [
      - \mathbf{K}_0(\rho_1 |\br - \br'|) + 2 i H_0^{(1)} (\rho_1 |\br - \br'|)
      \right ]  \\ + \frac{1}{4} \sum_{j=2}^5 e_j \rho_j
    \bK_0(-\rho_j |\br - \br'|) \; ,
  \end{multline}
  for real-valued $\betan$.
  \end{theorem}

\begin{proof}
The formula can be derived using Fourier analysis.
  See~\cref{app:greensfun} for details.
\end{proof}

\subsection{Variable coefficients}
\label{sec:Lippmann}

Recall that $\alpha(\br) = \alpha_0 + \alpha_c(\br)$, 
$\beta(\br) = \beta_0 + \beta_c(\br)$, where $\alpha_c$ and $\beta_c$ are smooth functions which are compactly supported on $\Omega \subset \mathbb{R}^2$. We represent $\sigma$ using
\begin{equation} \label{sigmadef}
 \sigma(\br) = \int_{\Omega} \GS(\br,\br') \mu(\br') \, \dd A(\br') \; ,
\end{equation}
where $\GS$ is the Green's function from \cref{thm:greensfun}
for the constant parameters $\alpha_0$, $\beta_0$, and $\gamma$.
Substitution of this representation into \cref{eq:integrodiffsigma}
one obtains the adjoint Lippmann-Schwinger equation:
\begin{equation}
  \frac{\alpha}{\alpha_0} \mu(\br) + \frac{1}{2} \sum_{i=1}^8 \int_{\Omega} K_i(\br,\br') \mu(\br') \, \dd A(\br')  =  f(\br) \; , \label{eq:LS}
\end{equation}
where the kernels $K_i \ (i = 1, ... ,8)$ in the integral equation above are defined by
\begin{align}
    &K_1(\br,\br') = 2 \, \partial_x  \alpha_c(\br)  \, \partial_x  \Delta_\br \GS(\br,\br') \; , \label{kernel1} \\
    &K_2(\br,\br') = 2 \, \partial_y  \alpha_c(\br) \, \partial_y  \Delta_\br  \GS(\br,\br') \; , \label{kernel2} \\
    &K_3(\br,\br') = \Delta_\br  \alpha_c (\br) \, \Delta_\br  \GS(\br,\br') \; ,  \label{kern3}  \\
    &K_4(\br,\br') =  - (1-\nu) \partial^2_x \alpha_c (\br) \, \partial_y^2 \GS(\br,\br') \; , \label{kern4}\\
    &K_5(\br,\br') =  - (1-\nu) \partial^2_y \alpha_c (\br) \, \partial^2_x\GS(\br,\br') \; , \label{kern5}\\
    &K_6(\br,\br') = 2 (1-\nu)\partial_{xy} \alpha_c (\br) \,  \partial_{xy} \GS(\br,\br') \; , \label{kern6}\\
    &K_7(\br,\br') =  \frac{ \alpha_c(\br) \beta_0 - \alpha_0 \beta_c(\br)}{\alpha_0 } \GS(\br,\br') \; , \label{kern7}\\
    &K_8(\br,\br') = - \frac{2 \gamma \alpha_c (\br) }{\alpha_0 } G_\phi (\br,\br') \; , \label{kernel8} 
\end{align}
and the righthand side is given by
\begin{align}
    &f(\br) = -(\alpha\Delta_{{\mathrm S}}^2 - \beta) \partial_z \phi^\inc - \gamma \phi^{\rm inc} \; .\label{rhs}
\end{align}
If $\phi^\inc$ satisfies the constant coefficient part of the PDE, then we are left with the righthand side:
\begin{align*}
    f(\br) &= - (2 \nabla \alpha_c \cdot \nabla\Delta + (\Delta \alpha_c) \Delta + (1-\nu) (2 \partial_{xy} \alpha_c \partial_{xy} - \partial_{x}^2 \alpha_c\partial_y^2 -  \partial_{y}^2 \alpha_c\partial_x^2 ) ) \partial_z \phi^\inc \\ 
    &\qquad \qquad - ( \tfrac{\alpha_c}{\alpha_0}\beta_0 - \beta_c ) \partial_z \phi^\inc +\tfrac{ \alpha_c}{\alpha_0}  \gamma    \phi^\inc \, .
\end{align*}
Then we use the following representation(s) to obtain the total fields $\phi$ and $\partial_z \phi$ on the boundary $z = 0$:
\begin{align}
     \phi(\br) &= \phi^\inc(\br) + \int_{\bbR^2} G_{\phi}(\br,\br') \mu(\br') \, \dd A(\br') \, , \\
    \partial_z \phi(\br) &= \partial_z \phi^\inc(\br) + \frac{1}{2} \int_{\bbR^2} \GS(\br,\br') \mu(\br') \, \dd A(\br') \, ,
\end{align} 
{\hide{red} where $\mu$ is defined on all of $\bbR^2$ via extension by zero.}

\section{Analytical properties of the integral equation}\label{sec:analyticalie} 
{\hide{red}As above, we assume that} $\alpha_c, \beta_c \in C_c^\infty(\Omega)$ {\hide{red}} and that $\alpha > 0$ everywhere in the domain.
{\hide{red}The set $\Omega$, which must contain the supports of $\alpha_c$ and $\beta_c$, can be taken to be a 
bounded domain with smooth boundary, which simplifies the statements of certain technical lemmas below.}

\subsection{Mapping properties of the integral equation} 

In this section we show that \cref{eq:LS} is an {\hide{red}integral equation of the second kind.}
The main analytical result is the following theorem.
\begin{theorem}\label{compact}
  Let $m\geq 0$. 
  The integral operators associated with the kernels
  \cref{kernel1,kernel2,kern3,kern4,kern5,kern6,kern7,kernel8}
  map $H^m(\Omega) \to H^{m+1}(\Omega)$.
\end{theorem}
\begin{corollary} \label{cor:compact}
  The same operators are compact from $L^2(\Omega) \rightarrow L^2(\Omega)$.
\end{corollary}
\begin{corollary}
  Equation \cref{eq:LS} is an {\hide{red}integral equation of the second kind} in $L^2(\Omega)$. 
\end{corollary}

Before proving \cref{compact}, we require the following three lemmas.
The first lemma, appearing in Section 6.3 of \cite{evans}, pertains to the elliptic regularity of solutions to Poisson's equation and the second lemma
is a straightforward extension to solutions of the inhomogeneous biharmonic
equation. 
The third lemma, adapted from Proposition 4.4 of Chapter 4 of
\cite{taylor}, pertains to the compact embedding of Sobolev spaces
and establishes \cref{cor:compact}.
\begin{lemma}\label{ellipticregularity}
  Suppose that $f \in H^m(\Omega)$ and let $u \in H^1(\Omega)$ be a weak solution of the Poisson equation $$ \Delta u = f \quad x \in \Omega \, .$$
  Then, $u \in H^{m+2}_{\rm loc}(\Omega)$.
\end{lemma}
\begin{lemma}\label{ellipticregularity2}
  Suppose that $f \in H^m(\Omega)$ and let $u \in H^3(\Omega)$ be a weak solution of the inhomogeneous biharmonic equation $$ \Delta^2 u = f \quad x \in \Omega \, .$$
  Then, $u \in H^{m+4}_{\rm loc}(\Omega)$.
\end{lemma}
\begin{lemma}[Rellich]\label{rellich}
  For any $s \geq 0, \sigma > 0$ the inclusion map
  $$ j : H^{s+\sigma}(\Omega) \rightarrow H^{s}(\Omega) \; , $$
  is compact. 
\end{lemma}

We now proceed with the proof of \cref{compact}.
\begin{proof}[Proof of \cref{compact}]  
We present the details for the kernels \cref{kernel1,kernel2,kern3,kern4,kern5,kern6,kern7}
in the case that $\Im(\beta_0) = 0$. The arguments for $\Im(\beta_0) \ne 0$ and 
\cref{kernel8} are similar. 

  We first examine the asymptotics of\cref{kernel1,kernel2,kern3,kern4,kern5,kern6,kern7} as
  $|\br - \br'| \rightarrow 0$. Recall the following power series expansions
  for the special functions $\mathbf{H}_0(\rho  r )$, $Y_0(\rho  r )$, and $J_0(\rho r )$ for $\rho \in \mathbb{C}$ (see \cite{abramowitz1948handbook}): 
\begin{align}
\mathbf{H}_0\left(\rho  r \right) &= \sum_{k=0}^\infty \frac{(-1)^k (\frac12 \rho r)^{2k+1}}{(\Gamma(k+\frac32))^2} \; , \label{expansion1} \\
Y_{0}\left(\rho  r  \right)&=  \frac{2}{\pi} [ \ln(\tfrac12 \rho r) + {\hide{red} \gamma_E}] J_0(\rho r)  + \frac{2}{\pi} \sum^\infty_{k=1} (-1)^{k+1} H_k \frac{(\tfrac{1}{2}\rho r )^{2k}}{(k!)^2} \; ,  \\
J_{0}\left(\rho r \right) &= \sum_{k=0}^{\infty}\frac{ (-1)^{k} (
\tfrac{1}{2}\rho r)^{2k}}{k!\Gamma\left(k+1\right)} \label{expansion3} \; , 
\end{align}
where $\Gamma(x)$ is the gamma function, {\hide{red} $\gamma_E$} is the Euler-Mascheroni constant, and $H_k$ is a harmonic number. Recall that $\mathbf{K}_n = \mathbf{H}_n - Y_n$ and $H^{(1)}_n = J_n + i Y_n$. From this we can show that the Green's function \cref{eq:GS} has the power series expansion:
\begin{multline*}
  \GS(\mathbf{r},\br') = \Psi_0(\mathbf{r},\br') - \sum^{5}_{j = 1} e_j \rho_j^2  \left[ \frac{1}{2\pi}   \ln( |\br - \br'|^2)  J_0(\rho_j |\br - \br'|)
    \vphantom{\sum_{k=0}^\infty} \right. \\
    \left. +   \sum_{k=0}^\infty \frac{ (-1)^{k} (\frac12 \rho_j | \br - \br' |)^{2k+1} }{2(\Gamma(k+\frac32))^2} \right] \; ,
\end{multline*}
where $\Psi_0(\mathbf{r},\br')$ is a smooth function given by: 
\begin{multline*}
    \Psi_0(\mathbf{r},\br') = \frac{1}{\pi} e_1 \rho_1^2 [\ln(2/\rho_1) - \gamma + i \pi ] J_0 ( \rho_1 |\br - \br'|) \\ + \sum^{5}_{j = 2} e_j \rho_j^2 \frac{1}{\pi} [ \ln(-2/ \rho_j ) - \gamma] J_0(\rho_j |\br - \br'|) \\
    + \frac{1}{\pi} \sum^{5}_{j = 1} \sum^\infty_{k=1} e_j \rho_j^2 (-1)^{k} H_k \frac{(\tfrac{1}{2}\rho_j |\br - \br'| )^{2k}}{(k!)^2} \; .
\end{multline*}
Using the moment relations \cref{eq:momentrelation}, one finds that the leading order behavior of $\GS$ has the form:
\begin{multline*}
  \GS(\mathbf{r},\br') = \Psi_0(\br,\br') + \frac{1}{8\pi\alpha_0} |\br - \br'|^2
  \ln( |\br - \br'|^2)
  + |\br - \br'|^7 \Psi_1(\br,\br') \\
  + |\br - \br'|^6 \log(|\br - \br'|^2) \Psi_2(\br,\br') \; ,
\end{multline*}
where $\Psi_1$ and $\Psi_2$ are also smooth. Note that the leading order term
is a scalar multiple of the Green's function for the biharmonic equation
and the remaining terms have four continuous derivatives. 

Now define the volume potential $\mathcal{V}$ for the biharmonic equation
in two dimensions as the integral operator
$$ \mathcal{V}[\mu](\br) := \frac{1}{16\pi} \int_\Omega |\br-\br'|^2 \log(|\br -
\br'|^2) \mu(\br') \, \dd A(\br') \, ,$$ 
where $\mu \in H^m(\Omega) $.  Let $ u(\br) := \mathcal{V}[\mu](\br)$.
By \cref{ellipticregularity2}, it must be that
$u \in H^{m+4}_{\textrm{loc}}(\Omega)$. By the decomposition of $\GS$ above,
{\hide{red} the integral operator with $\GS$ as its kernel then maps
$H^m(\Omega)\to H^{m+4}_{\textrm{loc}}(\Omega)$, so that} integral operators defined for kernels with up to three derivatives
of $\GS$ {map $H^m(\Omega)$ to $H^{m+1}_{\textrm{loc}}(\Omega)$}.

The kernels \cref{kernel1,kernel2,kern3,kern4,kern5,kern6,kern7} are
each given as sums of kernels of the form
$K(\br,\br') = \psi(\br) \tilde{K}(\br,\br')$
where $\psi(\br)$ is a smooth, compactly supported function in $\Omega$
and $\tilde{K}$ consists of up to third order derivatives of $\GS$.
By the preceding, the integral operator for each kernel then maps
$H^m(\Omega) \to H^{m+1}(\Omega)$. The argument for kernel \cref{kernel8}  follows similarly. 
\end{proof}

\subsection{Properties of the solutions of the integral equation}

In this section, we show that solutions $\mu$ of the integral
equation \eqref{eq:LS} have the same regularity as the data $f$, and that the surface density $\sigma$
and its derivatives up to order four are as regular. We then show that $\sigma$ satisfies certain
decay and radiation conditions and that it is a solution of the
original integro-differential equation~\cref{eq:integrodiffsigma}. 
\begin{theorem}
  \label{thm:solregularity}
  Suppose that $m \geq 0$ and $f \in H_0^m(\Omega)$ and that $\mu \in L^2(\Omega)$
  is a solution of \cref{eq:LS}. Then $\mu \in H_0^m(\Omega)$.
  Further, $\sigma$ defined as in \cref{sigmadef} has $\sigma \in
  H^{m+4}_{\textrm{loc}}(\bbR^2)$. 
\end{theorem}
\begin{proof}
  Re-arranging \cref{eq:LS}, we have
  
  $$   \mu(\br) = \frac{\alpha_0}{\alpha} \left ( f(\br)
  - \frac{1}{2}  \sum_{i=1}^8
  \int_{{\hide{red} \Omega}} K_i(\br,\br') \mu(\br') \, \dd A(\br') \right ) \; , $$
  so that $\mu$ is the sum of $f\in H^m(\Omega)$ and
  a function that is smoother than $\mu$. Applying induction,
  we have that $\mu \in H^{m}(\Omega)$. We observe also that the
  support of any output of the integral operators is contained in
  the support of $\alpha_c$ and $\beta_c$. Thus, $\mu \in H^m_0(\Omega)$.

  Since $\mu \in H^m_0(\Omega)$ it can be extended to
  $\mu \in H^m(U)$ for any bounded domain $U$ containing $\Omega$.
  Arguing as in the proof of \cref{compact}, the integral operator
  corresponding to the kernel $\GS$ maps $H^m(U) \to H^{m+4}(U)$.
  Thus, $\sigma \in H^{m+4}_{\textrm{loc}}(\bbR^2)$. 
\end{proof}

{\hide{red}
The re-arrangement from the proof of \cref{thm:solregularity} can be used
with bootstrapping to establish the following corollaries. The continuous case relies
on the iterated operators mapping $L^2(\Omega)\to H^2_0(\Omega)$ and the fact that 
an integral operator with a weakly singular kernel applied to a continuous function results in a continuous
function (see, e.g., Theorem 2.29 in \cite{kress}).}
\begin{corollary}
\label{cor:smoothness}
  Suppose that $f \in C_c^\infty(\Omega)$ and that $\mu \in L^2(\Omega)$
  is a solution of \cref{eq:LS}. Then, $\mu \in C_c^\infty(\Omega)$ and
  $\sigma$ defined as in \cref{sigmadef} has $\sigma \in C^\infty(\bbR^2)$. 
\end{corollary}
{\hide{red}
\begin{corollary}
\label{cor:continuous}
  Suppose that $f \in C_c(\Omega)$ and that $\mu \in L^2(\Omega)$
  is a solution of \cref{eq:LS}. Then, $\mu \in C_c(\Omega)$. 
\end{corollary} }

In the case that the plate has some dissipation added, then
$\sigma$ decays sufficiently fast that quantities like the
corresponding velocity potential can be defined using standard
integrals. For the non-dissipative case, $\sigma$ decays slowly
but is an oscillatory, outgoing solution and the velocity
potential defined by $S[\sigma]$ can be understood as a principal value integral.

{\hide{red} 
\begin{lemma}
    \label{lem:pvsinglelayer}
    Suppose that $\sigma$ is a continuous function on $\mathbb{R}^2$ and that,
    for some $k\in \bbR_+$,  

    $$ \sigma(\br) = \frac{\exp(i k |\br|)}{\sqrt{|\br|}} \left (
    \sigma_\infty( \theta) + O\left ( \frac{1}{|\br|} \right ) \right) \; , 
    \quad |\br|\to \infty \; ,$$
    where $\theta = \operatorname{Arg}(\br)$ and $\sigma_\infty$ is a continuous 
    function defined on $\bbR/2\pi$. 
    Then, $S[\sigma]$ is harmonic 
    for $z < 0$ and the usual jump conditions hold, i.e., for $\mathbf{r}_{\mathrm{3d}}$ on the plane $z = 0$,
    we have 

\begin{align*}
\lim_{h \to 0} S[\sigma](\mathbf{r}_{\mathrm{3d}} \pm h \hat{z}) &= S[\sigma](\mathbf{r}_{\mathrm{3d}}) \; \textrm{ and} \\
\lim_{h \to 0} \partial_z S[\sigma](\mathbf{r}_{\mathrm{3d}} \pm h \hat{z}) &=  \mp \frac12 \sigma(\br) \; ,
\end{align*}
where $\hat{z}$ is the unit vector in the positive $z$ direction.
\end{lemma}
\begin{proof}
We consider the convergence of \cref{eq:Sdef} in the principal value sense (cf. \cref{eq:pvdef})
as $R\to \infty$. For $|\brthree'|>|\brthree|$, we have the estimate~\cite{greengard1988rapid}
$$ \left | \frac{1}{|\brthree'|} - \frac{1}{|\brthree-\brthree'|} \right |
\leq \frac{1}{|\brthree'|-|\brthree|} \frac{|\brthree|}{|\brthree'|} \; .$$
Combining this with the assumptions on $\sigma$, we obtain 
$$ \frac{1}{4\pi|\brthree-\brthree'|}\sigma(\br') = 
\frac{1}{4\pi |\br'|} \frac{\exp(i k |\br'|) \sigma_\infty(\theta)}{\sqrt{|\br'|}} + 
O \left (|\br'|^{-5/2} \right ) \; , |\br'|\to \infty \; .$$
Let $R > 2|\brthree|$. Applying an integration by parts, we have 
\begin{multline*}
    \int_{2|\brthree| < |\br'| < R} \frac{1}{4\pi |\br'|} \frac{\exp(i k |\br'|) \sigma_\infty(\theta)}{\sqrt{|\br'|} }
    \, \mathrm{d} A(\br') \\ = 
   \frac{1}{4\pi} \int_0^{2\pi} \sigma_\infty(\theta) \, \mathrm{d}\theta \int_{2|\brthree|}^R \frac{\exp(i k r)}{\sqrt{r}} 
    \, \mathrm{d} r \\
    = \frac{1}{4\pi i k} \int_0^{2\pi} \sigma_\infty(\theta) \, \mathrm{d}\theta \left (\frac{\exp(i k R)}{\sqrt{R}} -
    \frac{\exp(2 i k |\brthree|)}{\sqrt{2|\brthree|}} + \frac{1}{2} \int_{2|\brthree|}^R \frac{\exp(i k r)}{r^{3/2}} 
    \, \mathrm{d} r \right ) \; .
\end{multline*}
From the above, we have convergence of \cref{eq:Sdef} in the sense of 
\cref{eq:pvdef} as $R\to \infty$ with a bound which decays like 
$1/\sqrt{R}$ and the constants depending only on the magnitude of $\brthree$. Thus,
convergence is locally uniform in $\bbR^3$. Simpler bounds imply that the
derivatives with respect to $\brthree$ also converge locally uniformly in $z<0$;
the bounds in that case depend on $|z|$ and $|\brthree|$.

For any fixed $R$, the integral 
$$ \int_{|\br'|<R} \frac{1}{4\pi|\brthree-\brthree'|} \sigma(\br') \, \mathrm{d}A(\br') $$
is a harmonic function of $\brthree$ in $z<0$. By interchanging the order of
differentation and the $R\to \infty$ limit, we obtain that $S[\sigma]$ is harmonic
there as well. The usual jump properties~\cite{kress} are based on local properties 
of the density and surface, so the finite $R$ analogues of the jump relations will hold 
for any $R> |\brthree|$. Interchanging the limits, we obtain the desired jump
relations.
\end{proof}
}
\begin{proposition}
  \label{prop:gsrhodecay}
  Suppose that 
  $\mu \in L^2(\Omega)$ and $\sigma$ is defined as in \cref{sigmadef}.
  For $\Im(\beta_0) \ne 0$, we have 
  \begin{multline*}
    \sigma = O(1/|\br|^3) \\
    \textrm{and} \quad 
    \max \left ( | \nabla \sigma|,
    |\nabla \otimes \nabla \sigma|, | \nabla\otimes
    \nabla \otimes \nabla \sigma | \right )  =
    O (1/|\br|^4) \; , \quad |\br| \to \infty \; .
  \end{multline*}
  For $\Im(\beta_0) = 0$, we have

{\hide{red} \begin{equation}
   \sigma(\br) = \frac{\exp(i \rho_1 |\br|)}{\sqrt{|\br|}} \left (
    \sigma_\infty( \theta) + O\left ( \frac{1}{|\br|} \right ) \right) \; , 
    \quad |\br|\to \infty \; ,
  \end{equation}
      where $\theta = \operatorname{Arg}(\br)$, $\sigma_\infty$ is a continuous 
    function defined on $\bbR/2\pi$, and }$\rho_1$ is the positive root of the
    dispersion relation, as in \cref{lem:dispersionpoly}.
\end{proposition}

\begin{proof}
  These decay properties follow from the asymptotic behavior of
  the kernel $\GS$ and the fact that $\mu$ has
  bounded support. In particular, the Struve function
  $\mathbf{K}_0$ has the asymptotic expansion~\cite{abramowitz1948handbook}

  $$ \mathbf{K}_{0}\left(z\right)\sim\frac{1}{\pi}
  \sum_{k=0}^{\infty}\frac{\Gamma\left(k+\tfrac{1}{2}\right)(\tfrac{1}{2}z)^{-2k-1}}{\Gamma\left(\tfrac{1}{2}-k\right)} \; ,$$
  which is valid for large $|z|$. 
  Applying the moment relations~\cref{eq:momentrelation}, we see that

  $$ \GS(\br,\br') = O(1/|\br|^3) \; , $$
  for $\Im(\beta_0) \ne 0$, {\hide{red} and thus $\sigma(\br) = O(1/|\br|^3)$ in that case. The decay for the derivatives can be
  derived similarly.

  For $\Im(\beta_0)=0$, we have instead}
  $$ \GS(\br,\br') = i e_1 \rho_1^{2}  H_0^{(1)} (\rho_1 |\br - \br'|) +
  O (1/|\br|^3) \; . $$
  {\hide{red} Because 
  $$ u(\br) = \int_{\Omega} H_0^{(1)}(\rho_1 |\br-\br'|)\mu(\br') \, \mathrm{d}A(\br') $$
  is an outgoing solution of the Helmholtz equation with wavenumber 
  $\rho_1$, the $\sigma$ defined by \cref{sigmadef} will have the claimed far-field behavior~\cite{colton2013integral}.}
\end{proof}

{\hide{red} Combining \cref{prop:gsrhodecay} and \cref{lem:pvsinglelayer} we have}
\begin{corollary}
\label{cor:ssigsolves}
    Suppose that $f \in L^2(\Omega)$ and $\mu \in L^2(\Omega)$ is a solution
    of \cref{eq:LS}. Then, $\sigma$ defined as in \cref{sigmadef} is a
    solution of \cref{eq:integrodiffsigma} and $S[\sigma]$ is a solution
    of the boundary value problem {\hide{red}\cref{eq:phipde1,eq:phibc1}}.
\end{corollary}

\section{Existence and uniqueness results}\label{sec:existenceuniqueness}

In this section, $\Omega$ is a bounded domain with smooth boundary
containing the support of $f$. As before, we suppose that {\hide{red}$\alpha(x,y) = \alpha_0 + \alpha_c(x,y)$ with
$\alpha_c \in C_c^\infty(\Omega;\bbR)$ and $\alpha_c+\alpha_0 >0$ and  that
$\beta(x,y) = \beta_0+\beta_c(x,y)$ with $\beta_c\in C_c^\infty(\Omega;\bbC)$ and $\Im(\beta_c)\ge 0$.} For 
some results in this section, we also require the following property: 
\begin{property}[Dissipative regime]
  \label{propy:dissconstants}
We say that the parameters are in the \textit{dissipative regime} if {\hide{red} $\Im(\beta) >0.$}
\end{property}

\subsection{A uniqueness result for the boundary value problem}

{\hide{red} Before proceeding to the integral equation reformulation, we establish the uniqueness
of solutions of the boundary value problem in the dissipative regime,
assuming appropriate decay conditions.}
\begin{proposition}
  \label{prop:pdeunique}
  Assume that the coefficient $\beta$ satisfies
  \cref{propy:dissconstants}.
  Let $\phi \in H^2_{\textrm{loc}}(z < 0)$ be a solution of the  homogeneous
  ($f=0$) version of {\hide{red} \cref{eq:phipde1,eq:phibc1}},  which has its normal trace in
  $H^4_{\textrm{loc}}(\bbR^2)$ and satisfies the following decay conditions
  \begin{align*}
    |\phi(\brthreet)| &= O (1/|\brthreet|) \; , & |\brthreet| \to \infty \; , \\
    \left | \frac{\brthreet}{|\brthreet|} \cdot
     \nabla_{\textup{3d} } \phi (\brthreet) \right | &= O (1/|\brthreet|^2) \; , & |\brthreet| \to \infty \; , \\
   |\partial_{z} \phi |  &=
    O (1/|\br|) \; , & z=0 \, , \quad |\br| \to \infty \; , \\
    \max \left ( | \nabla \partial_z \phi|,
    |\nabla \otimes \nabla \partial_z \phi|, | \nabla\otimes
    \nabla \otimes \nabla \partial_{z} \phi| \right )  &=
    O (1/|\br|^2) \; , & z=0 \, , \quad |\br| \to \infty \; .
  \end{align*}
  Then $\phi \equiv 0$. 
\end{proposition}



\begin{proof}[Proof of Proposition \ref{prop:pdeunique}]
  Suppose that $\phi$ satisfies the conditions of the theorem. 
  For given $R > 0$, let $\Omega_R =
  \{ \brthree = (x,y,z) : |\brthree| < R \textrm{ and } z < 0 \}$. 

  Let $\partial \Omega_R = D_R \bigcup H_R$ where $D_R$
  is the disc of radius $R$ in the plane $z=0$ and $H_R$ is the
  lower hemisphere of radius $R$.
  
  Then
  \begin{align*}
    0 &= \int_{\Omega_R} \phi \Delta_\textup{3d} \bar{\phi} - \bar{\phi}
    \Delta_\textup{3d} \phi \, \dd V \\
    &= \int_{H_R} \phi \partial_n \bar{\phi} - \bar{\phi} \partial_n \phi \, \dd S +
    \int_{D_R} \phi \partial_z \bar{\phi} - \bar{\phi} \partial_z \phi \, \dd A \; .
  \end{align*}
  The decay conditions imply that the first integral tends to zero
  as $R\to \infty$. 
  
  For $z=0$, we have
  $\phi = -\alpha/\gamma \Delta_{\rm S}^2 \partial_z \phi + \beta/\gamma \partial_z \phi$.
  Thus, the integral over $D_R$ becomes 

  \begin{align*}
    \int_{D_R} \phi \partial_z \bar{\phi} - \bar{\phi} \partial_z \phi \, \dd A
    &= -\frac{1}{\gamma} \int_{D_R} (\alpha \Delta_{\rm S}^2 -\beta)\partial_z\phi
    \partial_z\bar{\phi} - (\alpha \Delta_{\rm S}^2 -\bar\beta)\partial_z\bar\phi \partial_z \phi \, \dd A \\
    &= \frac{2i}{\gamma} \int_{D_R} {\hide{red}\Im(\beta)}|\partial_z\phi|^2 \, \dd A
    - \frac{1}{\gamma} \int_{D_R} \alpha  \partial_z\bar \phi  \Delta_{\rm S}^2 \partial_z \phi
    - \alpha \partial_z \phi  \Delta_{\rm S}^2 \partial_z \bar\phi  \, \dd A \; .
  \end{align*}
  {\hide{red} Integration by parts in the second integral gives: 
  \begin{multline*}
      \int_{D_R} \alpha  \partial_z\bar \phi  \Delta_{\rm S}^2 \partial_z \phi
    - \alpha \partial_z \phi  \Delta_{\rm S}^2 \partial_z \bar\phi  \, \dd A = \\
    \alpha_0 \int_{\partial D_R}  \partial_z \bar \phi  \partial_n \Delta \partial_z \phi - \partial_n \partial_z \bar \phi \Delta \partial_z \phi - \partial_z  \phi  \partial_n \Delta  \partial_z \bar \phi + \partial_n \partial_z  \phi \Delta \partial_z \bar \phi  \, \dd S \; ;
  \end{multline*}
  see, e.g., Chapter II \S 12 of \cite{landau}.
Derivatives of $\alpha$ arising in the boundary terms are zero since $\partial D_R$ is outside of $\Omega$. It follows from the assumptions on $\phi$ that the above integral vanishes in the limit} $R\to \infty$.

  We then obtain that

  {\hide{red}$$ 0 = \frac{2}{\gamma} \int_{D_R} \Im(\beta)|\phi_z|^2 \, \dd A \; , $$}%
so that $\phi_z \equiv  0$ on the surface $z=0$. But then the
  boundary condition on $z=0$ implies that $\phi = 0$ there as well.
  We see that $\phi$ is a bounded solution of the half-space Dirichlet
  problem and is thus $0$ everywhere. 
\end{proof}

\subsection{Invertibility of the integral equation}

The main results of this section establish the invertibility of the integral
equation \cref{eq:LS}, and thus the existence of solutions of the boundary value problem
\cref{eq:phipde1,eq:phipde2}, in the dissipative regime. Without dissipation, i.e. for real-valued $\beta_0$,
we show that the integral equation \cref{eq:LS} is invertible except, possibly, for a set of nowhere dense values of $\beta_0$.
{\hide{red} \begin{theorem}
  \label{thm:inteqexistunique}
  Suppose $\Omega$ is a bounded domain with smooth boundary
  and that $\beta$ satisfies \cref{propy:dissconstants}.
  Let $f\in L^2(\Omega)$ be given. Then, the integral equation
  \cref{eq:LS} has a unique solution.
\end{theorem}

{\hide{red}The above theorem, combined with \cref{cor:ssigsolves}, implies existence of solutions of the associated boundary value problem, assuming \cref{propy:dissconstants} holds.}

\begin{corollary}
 Suppose $\Omega$ is a bounded domain with smooth boundary
  and that $\beta$ satisfies \cref{propy:dissconstants}.
  Let $f\in L^2(\Omega)$ be given. Then, the boundary value problem
  {\hide{red} \cref{eq:phipde1,eq:phibc1}} has a unique solution satisfying
  the decay conditions in \cref{prop:pdeunique}. 
\end{corollary}

{\hide{red} Applying the analytic Fredholm theory, we will also use 
\cref{thm:inteqexistunique} to establish the following.}
\begin{corollary}
  \label{cor:inteqrealcase}    
    Suppose $\Omega$ is a bounded domain with smooth boundary and that $\beta_0$ is real.
  Let $f\in L^2(\Omega)$ be given. Then, the integral equation
  \cref{eq:LS} has a unique solution, except, possibly, for
  a set of nowhere dense values of $\beta_0$.
\end{corollary}

Again, \cref{cor:ssigsolves} then implies
\begin{corollary}
 Suppose $\Omega$ is a bounded domain with smooth boundary and that $\beta_0$ is real.
  Let $f\in L^2(\Omega)$ be given. Then, there exists a solution to the boundary value problem
  {\hide{red} \cref{eq:phipde1,eq:phibc1}} except,
  possibly, for a set of nowhere dense values of $\beta_0$.
\end{corollary}
}

Before proving \cref{thm:inteqexistunique}, we require {\hide{red}the following preliminary result concerning} the decay of single
layer potentials for densities defined on the plane $z=0$ that
satisfy certain decay conditions.
\begin{lemma}
  \label{prop:ssigma}
  Suppose that $\sigma$ is continuously differentiable {\hide{red}with}
  $$  \sigma(\br) = O(1/|\br|^3) \quad \textrm{and} \quad
  |\nabla \sigma(\br)| = O(1/|\br|^4) \; .$$
  Then,   
  $$  S[\sigma] (\brthreet) = O(1/|\brthreet|) \textrm{ and } 
  \left | \frac{\brthreet}{|\brthreet|} \cdot \nabla_{\threedt} S[\sigma] (\brthreet) \right|
  = O(1/|\brthreet|^2) \; \textrm{ as } |\brthreet| \to \infty .$$
  
  \begin{proof}
    Let $\sigma$ satisfy the assumptions above. For $R > 0$, let
    $D_R$ denote the disc of radius $R$ in $\bbR^2$. Let
    $G(\brthree,\brthree') = 1/(4\pi | \brthree-\brthree'|)$. 
   Let $\brthree \neq 0$ in $\bbR^3$ and
    let $R = |\brthree|/2$.

    We divide the integrals over $\bbR^2$ in the definition of $S[\sigma]$ and $\nabla_{\textrm{3d}} S [\sigma]$
    into three parts: the disc $D_R$, the annulus $A_{R,4R} = D_{4R}\setminus D_{R}$, and
    the disc exterior $E_{4R} = \bbR^2\setminus D_{4R}$. Let $n(\br)$ denote
    the outward normal at $\br \in \partial A_{R,4R}$.

%

    The decay conditions on $\sigma$ imply that $\sigma \in L^1(\bbR^2)$.
    For the contribution over $D_R \cup E_{4R}$, {\hide{red} we note that if $\br'\in D_R\cup E_{4R}$, then
    $R\leq|\brthree-\brthree'|.$ A direct calculation implies that
    \begin{multline*}
     \left | \int_{D_R\cup E_{4R}} G(\brthree,\brthree') \sigma(\br') \, \dd A(\br') \right | \le \frac{1}{4 \pi R} \|\sigma\|_{L^1(\mathbb{R}^2)}
     \quad \textrm{and} \\
    \left | \int_{D_R \cup E_{4R}} \nabla_{\threed} G(\brthree,\brthree') \sigma(\br') \, \dd A \right |
    \le \frac{1}{4 \pi R^2}\|\sigma\|_{L^1(\mathbb{R}^2)}\; .
    \end{multline*}}


    {\hide{red}Turning to the contribution from $A_{R,4R},$ we} note that $\sigma(\br') = O(1/R^3)$ for $\br' \in A_{R,4R}$. We also have
    that
    \begin{align*}
     \int_{A_{R,4R}} G(\brthree,\brthree') \, \dd A(\br')  &\leq
    \int_{A_{R,4R}} \frac{1}{4\pi |\br-\br'|} \, \dd A(\br') \\
    &\leq \int_{D_1} \frac{1}{4\pi |\br'|} \, \dd A(\br')
      + \frac{1}{4\pi} \int_{A_{R,4R}} \, \dd A(\br') \leq \frac{1}{2} + \frac{15}{4} R^2 \;.
    \end{align*}
     where $\br$ is the projection of $\brthree$ onto $z = 0$. Thus, the contribution over $A_{R,4R}$ is $O(1/R)$ for $S[\sigma]$. For the directional derivative, we split the estimate into two pieces.
    We have
    \begin{multline*}
      \frac{\brthree}{|\brthree|} \cdot \left( - \frac{\brthree-\brthree'}{4\pi|\brthree-\brthree'|^3} \right )
      \\
      = -\frac{z^2}{4\pi|\brthree||\brthree-\brthree'|^3} +
    \frac{\br}{|\brthree|} \cdot \nabla \frac{1}{4\pi\sqrt{z^2 + (x-x')^2 + (y-y')^2}}
    \; ,
    \end{multline*}
     where $\nabla = (\frac{\partial}{\partial x}, \frac{\partial}{\partial y})$. For the first piece, we have
    $$ \left | -\frac{z^2}{4\pi|\brthree||\brthree-\brthree'|^3} \right |
    \leq  \frac{1}{8\pi R} \frac{1}{|\brthree - \brthree'|}  \; .$$
    We can then use the same estimate for the integral of $G(\brthree,\brthree')$
    over the annulus to obtain that the contribution for this part is $O(1/R^2)$.

    For the second piece, we have that
    \begin{multline*}
      \int_{A_{R,4R}} \nabla \frac{1}{4\pi\sqrt{z^2 + (x-x')^2 + (y-y')^2}} \sigma(\br')
      \, \dd A(\br') \\
      = -\int_{A_{R,4R}} \frac{1}{4\pi\sqrt{z^2 + (x-x')^2 + (y-y')^2}} \nabla \sigma(\br')
      \, \dd A(\br') \\ + \int_{\partial A_{R,4R}}
      \frac{1}{4\pi\sqrt{z^2 + (x-x')^2 + (y-y')^2}}  \sigma(\br')  n(\br') \, \dd\ell(\br') \\
      = O(1/R^2) \; .
    \end{multline*}
  \end{proof}
\end{lemma}

The final preliminary result establishes the injectivity of
convolution with $\GS$ in the dissipative regime.
\begin{lemma}
  \label{prop:gsinjective}
  Suppose that $\beta$ satisfies \cref{propy:dissconstants}.  
  The operator $\mu \to \GS \ast \mu$ is injective from
  $L^2(\Omega)$ to $L^2(\bbR^2)$ for any bounded domain
  $\Omega \subset \bbR^2$.
\end{lemma}

\begin{proof}
  Any function $\mu \in L^2(\Omega)$ has a well-defined Fourier transform,
  $\hat{\mu}$,
  in $L^2(\bbR^2)$.
  The given operator has a Fourier multiplier which is
  bounded and continuous and its only zero is at the origin {\hide{red}(see \cref{eqn:hatsig}).} Thus, the Fourier
  transform of $\GS\ast \mu$ is zero only if $\hat{\mu}$ is
  zero and so the convolution is zero only if $\mu$ is zero. 
\end{proof}

We now proceed with the proofs of \cref{thm:inteqexistunique} and \cref{cor:inteqrealcase}. 
\begin{proof}[Proof of \cref{thm:inteqexistunique}]
Here we assume we are in the dissipative regime (see~\cref{propy:dissconstants}), 
which implies that $\Im(\beta_0)>0$. 
The equation \cref{eq:LS} is an integral equation of the second kind in $L^2(\Omega)$,
so that the existence of solutions can be reduced
  to uniqueness of solutions. Suppose that $\mu \in L^2(\Omega)$ is a solution
  of the homogeneous version of \cref{eq:LS}. Then $\sigma = \GS\ast \mu$
  solves the homogeneous version of the integro-differential equation \cref{eq:integrodiffsigma}
  and satisfies the decay conditions described in \cref{prop:gsrhodecay}.
  Applying \cref{cor:ssigsolves,prop:ssigma}, we obtain that $S[\sigma]$ is a solution of
  the original PDE that satisfies the conditions of \cref{prop:pdeunique}.
  Thus, $S[\sigma] \equiv 0$ in the lower half space. But in fact,
  $S[\sigma]$ is bounded throughout space by \cref{prop:ssigma} and by
  continuity satisfies the homogeneous Dirichlet boundary condition on $z=0$
  and is harmonic in the upper half space as well. Thus, $S[\sigma] \equiv 0$ in $\bbR^3$.
  By considering the difference of the limiting normal derivatives
  from the upper and lower half-spaces, we obtain that $\sigma \equiv 0$.
  But then $\GS\ast \mu \equiv 0$ and $\mu \equiv 0$
  by \cref{prop:gsinjective}. Applying the Fredholm alternative, we
  see that the integral equation has a unique solution for any
  $f \in L^2(\Omega)$. 
\end{proof}

\begin{proof}[Proof of \cref{cor:inteqrealcase}]
  Finally, we turn to the case of $\beta_0 \in \mathbb{R}.$ We begin by noting that the kernels $K_i$ in~\cref{eq:LS}, as functions of $\beta_0$, can be analytically continued from $\Im (\beta_0) >0$ to an open neighborhood of $\mathbb{R}\subset \mathbb{C}$;
  see \cref{prop:analyticbeta} for details. Moreover, the operators remain compact from $L^2(\Omega) \to L^2(\Omega).$ Thus, the operator on the left-hand side of~\cref{eq:LS} is Fredholm {\hide{red}index zero}, analytic, and invertible at a point in an open neighborhood of $\mathbb{R}$ and hence, by the Generalized Steinberg's theorem, has a bounded inverse for all $\beta_0$ in a neighborhood of $\mathbb{R}$ except possibly for certain isolated points \cite{gohberg,ammari}.
\end{proof}

\section{Numerical implementation and fast algorithms}\label{sec:numerics}

In this section, we describe the numerical discretization
and solution of the integral equation \cref{eq:LS}, establish the 
convergence properties of the method, and report tests that were 
performed to confirm the accuracy and convergence of the implementation.

\subsection{Discretization, quadrature rule, and fast solution}

We use a Nystr{\"o}m method to discretize \cref{eq:LS}.
In particular, the grid points are taken to be equispaced points of the 
form $\mathbf{x}_\mathbf{i}^h = (h\,i_1,h\,i_2)$ where $|i_1| < N_x$ and $|i_2| < N_y$ 
for some $N_x$ and $N_y$ sufficiently large that {\hide{red}$\Omega \subset [-hN_x,hN_x]\times [-hN_y,hN_y]$,} i.e.
that the grid covers the support of $\alpha_c$ and $\beta_c$. For such a grid,
it is well-known that the trapezoidal rule, i.e. 

$$ \int_\Omega g(\br) \, \dd A(\br) \approx \sum_{|i_1|<N_x, |i_2|<N_y} g(\mathbf{x}_{\mathbf{i}}^h) h^2 \; ,$$
has an error that decays super-algebraically for $g \in C_c^\infty(\Omega)$. 

Each kernel in \cref{eq:LS} can be written as $K_j(\br,\br')  = \beta_j(\br) w_j(\br-\br')$,
where $\beta_j \in C_c^\infty(\Omega)$ and $w_j$ is $\Gphi$, $\GS$, or a derivative of $\GS$. 
These kernels are weakly singular, so we must modify the trapezoidal rule to
obtain high-order accuracy. In particular, we seek a locally modified trapezoidal
rule that consists of a set of indices, $\mathcal{C}$, centered around
$(0,0)$ and a set of modified weights, $\alpha^j_{\mathbf{i}}(h)$, for each $K_j$ such that 
the quadrature rule

\begin{multline}
\label{eq:quadLS}
 \int_\Omega K_j(\bx^h_{\mathbf{i}},\br') \mu(\br') \, \dd A(\br') \\ \approx \beta_j(\bx_{\mathbf{i}}^h)
\left ( \sum_{\bf{i}-\bf{i}' \in \mathcal{C}} \alpha^j_{\mathbf{i}-\mathbf{i}'}(h) \mu(\bx_{\mathbf{i}'}^h) +
h^2 \sum_{\mathbf{i}-\mathbf{i}' \notin \mathcal{C}} w_j(\mathbf{x}_{\mathbf{i}}^h-\mathbf{x}_{\mathbf{i}'}^h) 
\mu(\mathbf{x}_{\mathbf{i}'}^h) \right ) \; ,
\end{multline}
has an error that is $O(h^p)$ for some reasonably large $p$. The 
weights should also satisfy a stability property, i.e. 

\begin{equation}
\label{eq:quadstab} 
\max_{j=1,\ldots,8} \sum_{\mathbf{i}\in\mathcal{C}} |\alpha^j_{\mathbf{i}}(h)| = O(h) \; .
\end{equation}

Such quadrature rules can be obtained using the Zeta-corrected theory developed in 
\cite{ wu2021corrected, wu2021zeta, unified, acha}. These methods compute 
integrals of the form
$$
   I = \int_\Omega \frac{g(\br)}{|\br-\mathbf{x}_\mathbf{i}^h|^{s}} \, \dd A(\mathbf{r}) \, ,  
   \qquad I = \int_\Omega g(\br) \log|  \br -\mathbf{x}_\mathbf{i}^h |  \, \dd A(\br),
$$
for $s \in \mathbb{R}$ and $g \in C_c^\infty (\Omega) $. In particular, when $g$ is a monomial, 
it is possible to show that the error in approximating the integrals above is given exactly by 
the Epstein-Zeta function (or derivatives thereof). A moment-fitting method can then be used to
obtain modifications of the trapezoidal rule that only affect a small number of grid points
near $\mathbf{x}_{\mathbf{i}}^h$, i.e. $|\mathcal{C}|$ is relatively small,
and compute the integrals to a given order of accuracy. While previous work has 
focused primarily on Laplace, Helmholtz, and Stokes kernels \cite{wu2021zeta,unified}, a straightforward
extension of this procedure can be used to obtain high-order quadratures for the biharmonic 
($|\br|^2 \log|\br|$) kernel and its derivatives up to order three (see code for more details). 
For the calculations in this paper, we use order $p=6$ for these corrections. 

\begin{remark}
As noted above, the Green's function $\GS$ can be expressed as a {\hide{red} smooth function}, plus a smooth 
multiple of the biharmonic kernel and a smooth multiple of $|\br|^7$. 
Owing to the symmetries of the quadrature rules, the stability property of the modified weights, and a 
further modification at $(0,0)$ to handle the constant term, a sufficiently high-order Zeta-corrected 
rule for the biharmonic kernel or any of its derivatives up to order three will result in a 
6th-order accurate rule for $\GS$ or its corresponding derivative. Hence, we employ
a 6th-order rule derived for the biharmonic kernel to discretize our integral operators.
\end{remark}

As confirmation of the numerical implementation of the quadrature rules, the quadrature corrections were verified numerically by applying the integral operators 
to a test density $\mu$ for various grid spacings $h$. The test 
density was chosen to be {\hide{red} $ \mu(x,y) = x e^{-\frac{x^2+y^2}{2}}\cos(\frac{1}{10}y+1)$} so that the density 
decays to machine epsilon within the domain of discretization {\hide{red} $[-10,10]^2$}. The integral was measured at a
fixed target  at another point {\hide{red} $(x,y) = (2,2)$} on the grid. In lieu of an analytic expression for the true value, we compare with reference values 
obtained using Matlab's adaptive integration routine. The results are shown in \cref{fig:intconv}.

\begin{remark}
  While the kernels found in the integral equation \eqref{eq:LS} are at most weakly singular, a naive implementation of the evaluation of these kernels will lead to catastrophic cancellation. This is due to the fact that the Hankel and Struve terms in the Green's function \eqref{eq:GS} possess $\log$ singularities which are canceled out by the moment relations \eqref{eq:momentrelation}. In the proof of \cref{compact}, it was shown that these $\log$ singularities cancel out and the leading order behavior of $\GS \sim r^2 \log r$. The effect of these cancellations is further exacerbated by the derivatives which are applied to the Green's function in the kernel, which would lead the $\log$ terms of the Hankel and Struve functions to become $O(\frac{1}{|\br|^3})$. To remedy this, stable evaluators were made for $\frac{i}{4}H^{(1)}_0(x) + \frac{1}{2\pi} \log(x)$ and $J_0(x) + i \mathbf{H}_0(x)$. For the Hankel part $(\frac{i}{4}H^{(1)}_0(x) + \frac{1}{2\pi} \log(x))$, this is done by evaluating the asymptotic expansions near zero \cref{expansion1}-\cref{expansion3} whenever the argument is small with the $\log$ term canceled explicitly. For the Struve part $(J_0(x) + i \mathbf{H}_0(x))$, evaluation was performed using generalized Gaussian quadrature in conjunction with the asymptotic series (see \cref{app:evaluation}). The derivatives of these functions up to order 3 were evaluated directly using the recurrence relations (\cref{app:greensderivatives}).
\end{remark}

\begin{figure}[ht]
    \centering
    \includegraphics[width=0.8\linewidth]{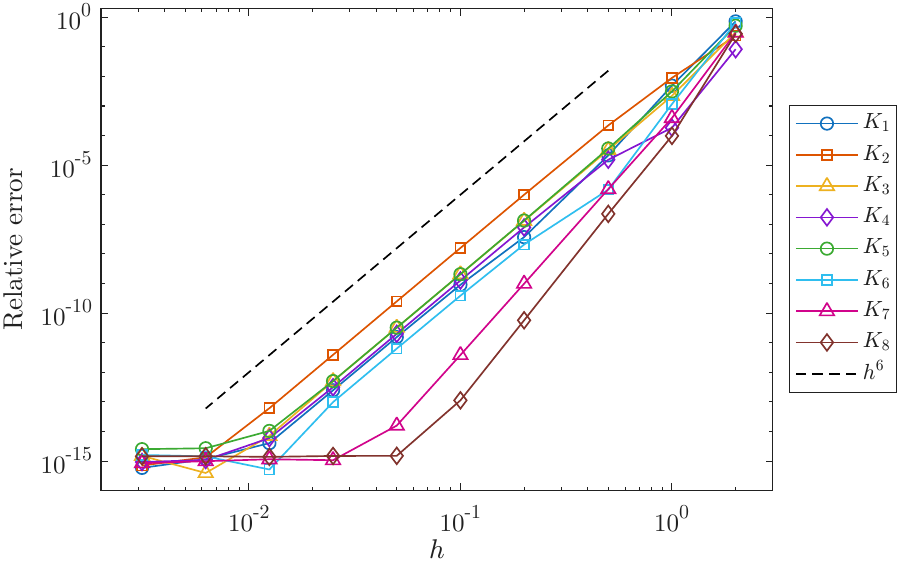}
    \caption{Convergence of discretized integral operators $K_1$-$K_8$ from equation \cref{eq:LS} computed using trapezoid rule with 6th-order Zeta corrections and compared against Matlab's adaptive integration routine.}
    \label{fig:intconv}
\end{figure}

Because the kernels, $w_j$, appearing in the integral operators are translationally invariant, 
the quadrature rule, \cref{eq:quadLS}, for $K_j$ is in the form of a discrete convolution (Toeplitz matrix product)
followed by diagonal multiplication by the $\beta_j(\bx^h_{\mathbf{i}})$ values. It is well-known 
that the FFT can then be used to reduce the complexity of performing this product from $O(N_x^2N_y^2)$ to 
$O(N_xN_y (\log N_x + \log N_y))$. This speed-up in the application of the discretized integral operators 
makes the system well-suited to solution by an iterative solver such as GMRES. This technique has previously 
been used in a number of other contexts, particularly in the solution of the Lippmann-Schwinger equation for 
the Helmholtz scattering problem; see, {\em inter alia}, \cite{duan2009high,gopal2022accelerated} and the
references therein.

\subsection{Convergence of the numerical method}
It can be shown that the solution of the discrete linear system
corresponding to the integral equation will converge to the solution of the
continuous integral equation with the same convergence rate as the quadrature
rule. Specifically, we have the following proposition.

\begin{proposition}\label{prop:convrate}
    Let $f \in C_{c}^\infty(\Omega)$ and let $\mu$ be the solution to \cref{eq:LS}. 
    Furthermore, suppose $\mu^h$ is the solution to the discretized equation using a $p$-th 
    order modified trapezoidal rule on a grid with spacing $h$, as in
    \cref{eq:quadLS}, to discretize the integral operators and suppose that 
    the quadrature rule satisfies the stability condition \cref{eq:quadstab}. 
    Let $\{\bx^h_{\mathbf{i}}\}$ denote the corresponding discretization nodes. 
    If the operator on the left-hand side of \cref{eq:LS} has a bounded inverse, 
    then there exists an $h_0>0$ such that  the discrete system is invertible for all $0<h<h_0.$ 
    Moreover, there exists a constant $C$ such that 
    $$\max_{\mathbf{i}} | \mu(\bx^h_{\mathbf{i}})-\mu^h_{\mathbf{i}}| \le C h^p \; ,$$
    for all $0<h<h_0.$
\end{proposition}
\begin{proof}
    Invertibility of the discrete system follows directly from \cref{thm:semidiscinvertible} and \cref{cor:fulldisc} {\hide{red} (cf. \cref{thm:compactC,label:SKIE_C})}, while the rate of convergence follows from \cref{thm:convergence}, proved in the appendix. 
\end{proof}

To check the accuracy of the solutions to the integral equation we apply two different tests. 
In both tests, we send a plane wave with wavenumber $k = 1.05$ toward a Gaussian thickness profile 
with width $\sigma = 4 $ m and height $A = 2 $ m. In the first test, we solve  \cref{eq:LS} for the density $\mu$ {\hide{red} using the generalized minimal residual method (GMRES)} and then use the FFT to interpolate the density onto a fine grid ($h_{\textrm{fine}}=0.0625$). The interpolated density is then used to compute the solution $\phi$ on the same grid. After computing the solutions for different $h$ on the fine grid, we can use high-order finite differences to check the consistency of the solution with the original PDE boundary condition \eqref{eq:phipde2}. In particular, the residual of the boundary condition was computed at the surface points $\{(m/2,n/2) | -8 \leq m,n \leq 8 \}$ and divided by the largest term in the boundary condition in order to get the relative errors. The $\ell^2$ and $\ell^\infty$ norms of these relative errors are displayed in \cref{fig:conv2}.  For the second test, we compute the {\hide{red}self-convergence} error in the integral equation by comparing the density interpolated on an even finer grid ($h_{\textrm{finest}}=0.015625$) with the converged density using both  $L^2$ and $L^\infty$ norms. Because this test does not rely on finite differences, it continues to converge even after the finite difference test has stalled. {\hide{red} For all the discretizations tested, GMRES required 34-48 iterations to converge to a tolerance of 1e-12.} {\hide{red} }

\begin{figure}[ht]
    \centering
    \includegraphics[width=0.8\linewidth]{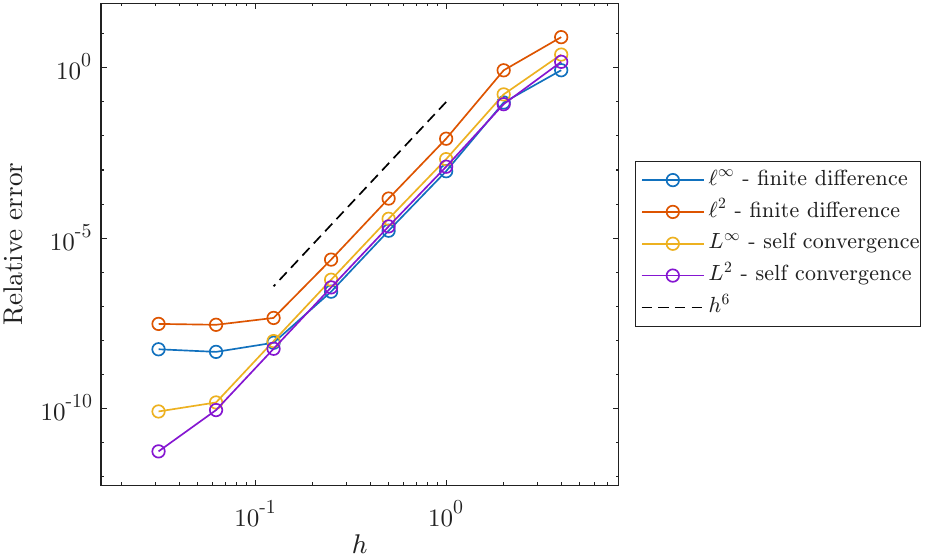}
    \caption{Convergence of the solution to the integral equation \cref{eq:LS} for a Gaussian thickness profile and incident plane wave data. The consistency of the solution with the surface PDE \cref{eq:phipde2} is plotted in red and blue, while the self-convergence of the density $\mu$ is plotted in yellow and purple.  }
    \label{fig:conv2}
\end{figure}

\section{Examples and applications}\label{sec:examples}

In this section, we present some numerical examples of flexural-gravity wave scattering problems that are inspired by glaciology. {\hide{red} Each of the examples was solved using GMRES, where the tolerance was set to 1e-8 unless stated otherwise}. Consistency with the PDE was also checked using finite differences, and the relative error in each case was measured to be around 1e-6. 

\begin{figure}
    \centering
    \includegraphics[width=0.85\linewidth]{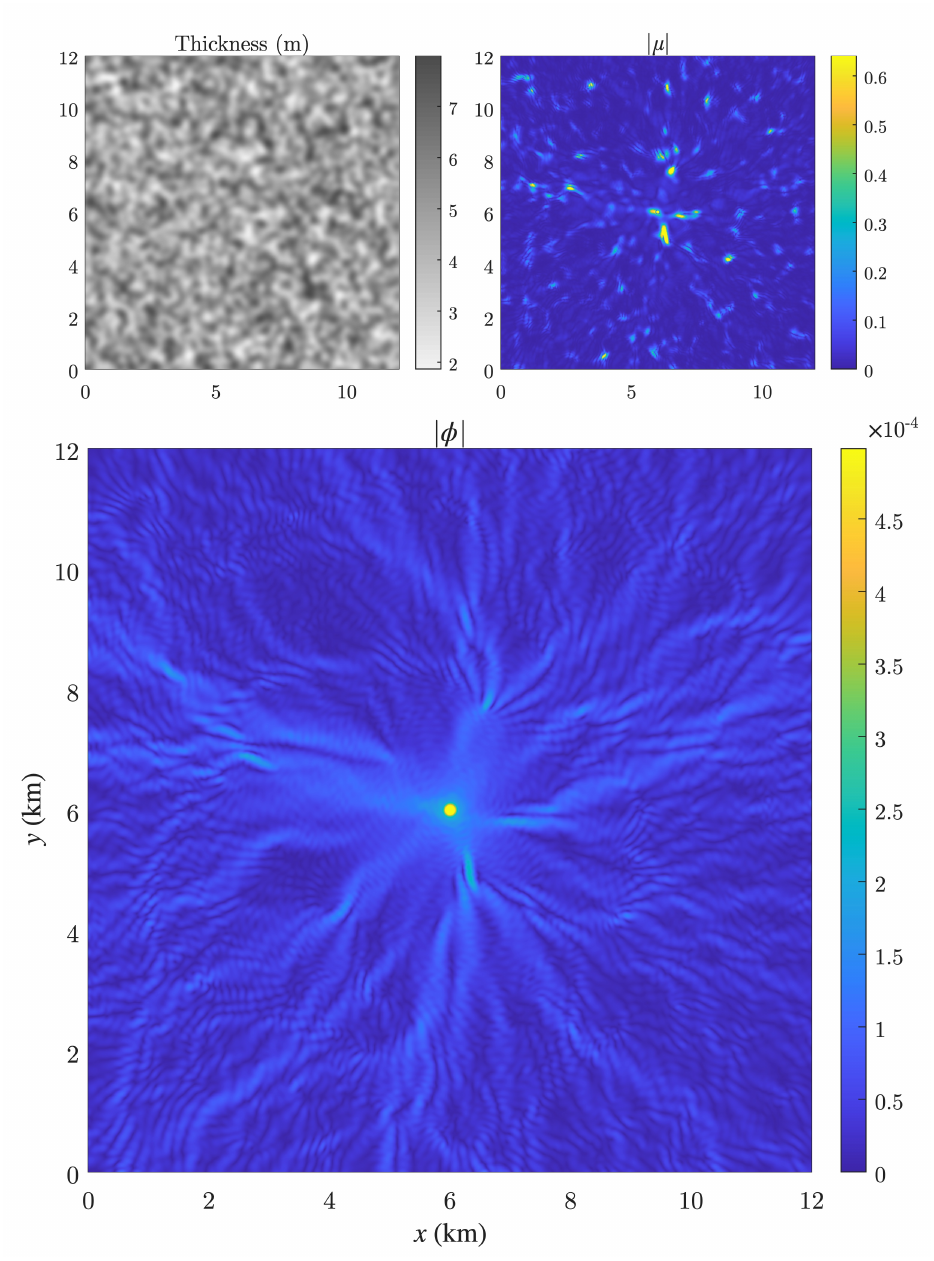}
    \caption{A solution for a random thickness profile with a Gaussian source. Top left: thickness profile generated by a sum of Gaussians with random heights. The average thickness is 5 m. Top right: the absolute value of the density $\mu$. Bottom: the absolute value of the velocity potential radiating outwards from the center of the domain. {\hide{red} The upper limit of the colorbar is set to half of the maximum amplitude.} }
    \label{fig:pointsrc}
\end{figure}

In the first example, we look at what happens to a localized source inside of a random medium, which is inspired by the highly heterogeneous nature of sea ice and ice shelves (\cref{fig:pointsrc}). To generate the random medium, a collection of Gaussians are placed at a regular interval and added to a mean thickness profile of 5 m. The amplitude of each Gaussian is sampled from a uniform distribution on the interval $[-1,1]$, the standard deviation of each Gaussian is chosen to be 140 m, and each Gaussian is placed one standard deviation apart on a 12 km by 12 km domain. Then, the right-hand side is chosen to be a Gaussian with unit amplitude and standard deviation of 50 m. {\hide{red} The solution to this problem converges to a relative residual of 1e-8 after 1660 iterations of GMRES.} As the wave energy propagates through the random medium, the wave begins to branch out in different directions; this wave branching phenomenon has previously been reported in \cite{jose2022branched,jose2023branched,flexbie}. The resulting interference pattern is produced by diffuse reflection through a highly heterogeneous medium.

\clearpage 

Next, we examine the effect of sinusoidal variations in thickness on the propagation of flexural waves (\cref{fig:rolls1}). Such undulations can be found in the surface rolls of Arctic ice shelves \cite{nekrasov2023rolls} and the rumples of Antarctic ice shelves \cite{collins1985creep}. These rolls are believed to be formed by compressive wrinkling of a thin ice layer, though the exact mechanism is largely unknown. For this study, the thickness $H(x,y)$ is based on the amplitude of the surface roll pattern $A$ and the observed freeboard $H_f$:
$$ H(x,y) = \frac{\rho_{\textrm{sea}}}{\rho_{\textrm{sea}} - \rho_{\textrm{ice}}} \left[ \left(\frac{A}{2} + \frac{A}{2}\sin(2 \pi x/ w) \right) \psi(x,y) + H_f \right] \, , $$
where $w$ is the {\hide{red}period of the rolls in meters} and $\psi(x,y)$ is a smooth tapering function given by
$$\psi(x,y) = {\hide{red} \frac14} (\erf( s(x - x_0) ) + \erf(s(x_1 - x)) ) (\erf( s(y - y_0) ) + \erf(s(y_1- y)) ) \; , $$
This function is almost unity on the region $[x_0,x_1] \times [y_0,y_1]$ and smoothly decays outside this region, ensuring that the pattern is compactly supported. The parameter $s$ controls the rate at which the pattern transitions between one and zero. For this study, we consider realistic values $A = 0.75 \ {\rm m}, w = 333.3 \ {\rm  m}, H_f = 0.35 \ {\rm  m}, s = 0.008{\hide{red} \ {\rm  m}^{-1}}$. We send in horizontal plane waves that sweep a range of wavenumbers between $0.001-0.06 \ \rm{ m}^{-1}$ and measure the amplitude of the solution at two different points before and after the roll pattern as a proxy for the reflected and transmitted fields (\cref{fig:rolls1}). {\hide{red} Due to the large number of solves required for this study, the GMRES tolerance was set to 1e-6.}

This frequency sweep demonstrates the degree to which the qualitative wave behavior is sensitive to the incident wavenumber. For both low and high wavenumbers, the incident wave is mostly transmitted with very little reflection. At times, the amplitude of the transmitted field exceeds one due to lateral scattering effects at the top and bottom boundaries of the roll geometry. At intermediate wavenumbers, the reflection and transmission vary dramatically, with the reflected component shooting up at intermittent values of $k$. This type of variation in the reflection is sometimes referred to as a resonance comb \cite{squire2001region}. There are a few wavenumbers where the field is almost entirely reflected, owing to a phenomenon known as Bragg scattering. One of these fields is plotted in \cref{fig:rolls2}, contrasted against a similar wavenumber for which the field is mostly transmitted. {\hide{red} The number of GMRES iterations required to solve these two problems are 635 and 788, respectively.}

\clearpage

\begin{figure}[h]
    \centering
    \includegraphics[width=0.9\linewidth]{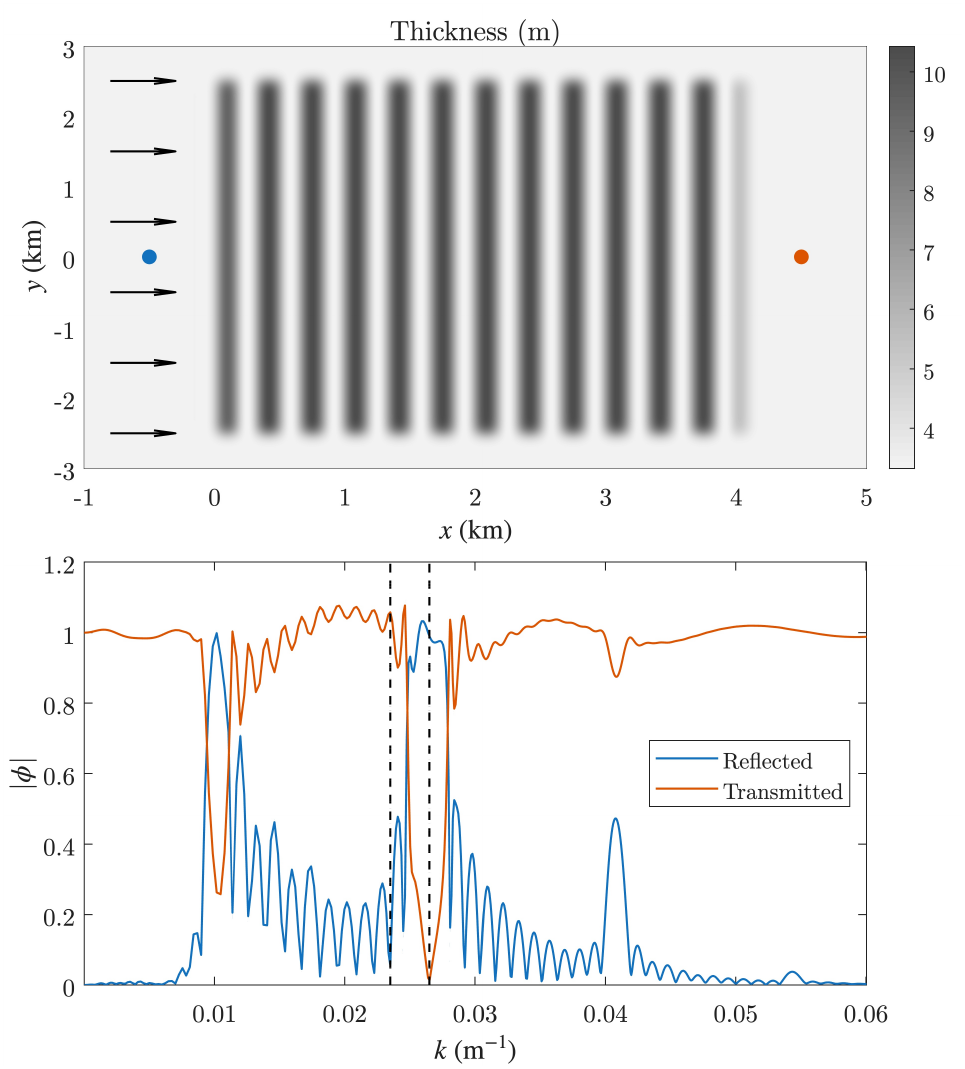}
    \caption{Top: surface roll pattern with a {\hide{red} period} of 333 m.  Arrows indicate direction of the incident plane wave. The blue dot indicates where the scattered field was measured as a proxy for the amplitude of the reflected wave, while the red dot marks the location where the total field was measured for the amplitude of the transmitted wave. Bottom: amplitudes of the reflected and transmitted fields for a range of wavenumbers. Dashed lines represent wavenumber values that are plotted in \cref{fig:rolls2}.}
    \label{fig:rolls1}
\end{figure}

\begin{figure}
    \hide{red} 
    \centering
    \includegraphics[width=0.8\linewidth]{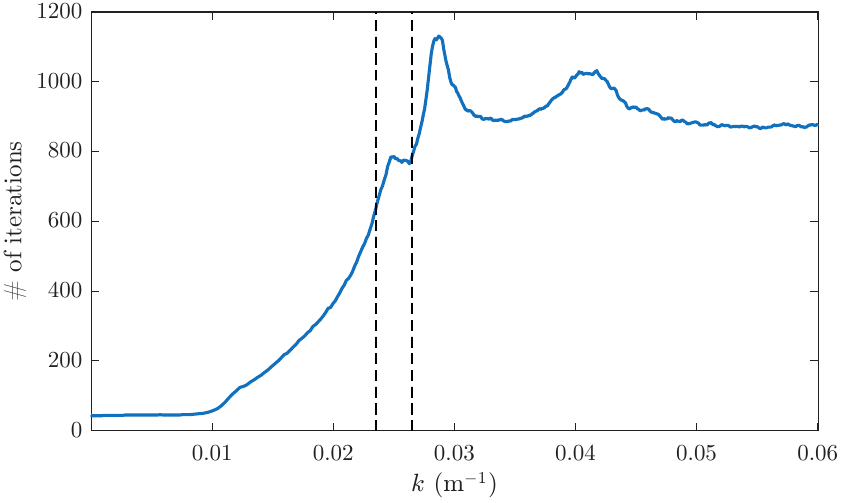}
    \caption{Number of GMRES iterations required to solve the surface roll problem in \cref{fig:rolls1} with tolerance set to 1e-6. The number of iterations grows rapidly in the $k = 0.01 - 0.03 \ {\rm m}^{-1}$ range due to scattering effects. Dashed lines represent wavenumber values that are plotted in \cref{fig:rolls2}. }
    \label{fig:rolls_gmres}
\end{figure}

\begin{figure}[h]
    \centering
    \includegraphics[width=\linewidth]{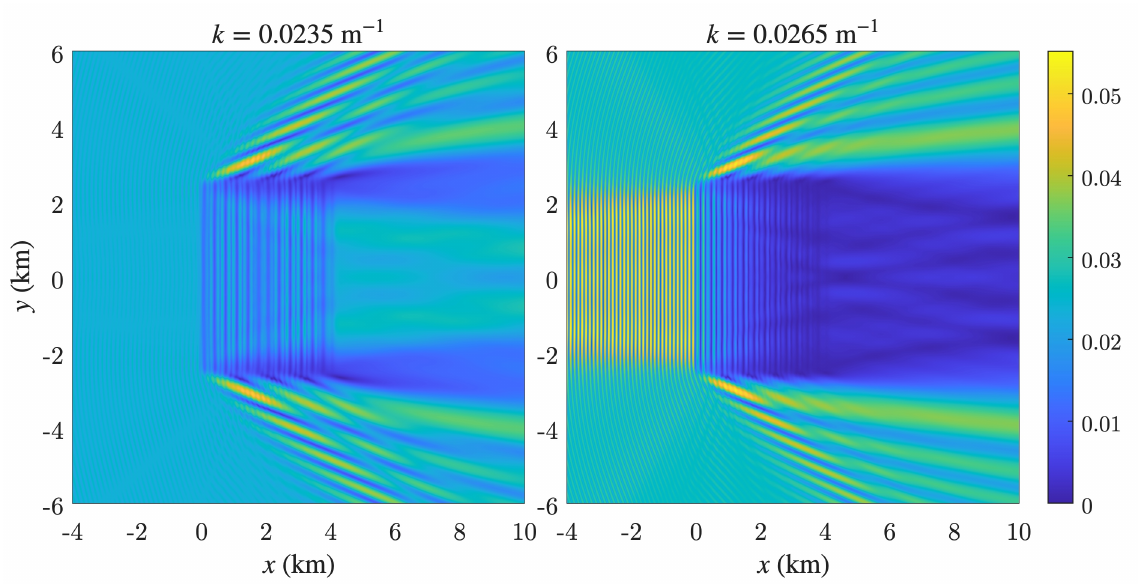}
    \caption{Wave scattering by an array of surface rolls for two nearby wavenumbers. The quantity that is plotted is the absolute value of the velocity $|\partial_z\phi|$. Left: a solution where most of the wave energy is transmitted. Right: a solution where most of the wave energy is reflected. }
    \label{fig:rolls2}
\end{figure}

We now turn to an example of scattering by a system of {\hide{red}{\it pressure ridges} }(\cref{fig:ridges5}). In contrast with the previous example, these ridges are not parallel but are allowed to branch outward in different directions. In this example, we create these ridges by convolving the Gaussian kernel with a constant density defined on some piecewise curves. Getting the derivatives of this thickness profile simply amounts to taking derivatives of the Gaussian kernel. We allow these ridges to have a thickness of 3 m while the background thickness is 1 m. Then, we send an incident plane wave with a wavenumber of $k = 0.105 \ \rm{m}^{-1}$. {\hide{red} The solution to this problem converges in 144 GMRES iterations.} The wave is scattered upon contact with each successive branch of the ridge pattern, suggesting that wave-ice interaction is significantly complicated by the presence of ridges.

\begin{figure}[h]
    \centering
    \includegraphics[width=\linewidth]{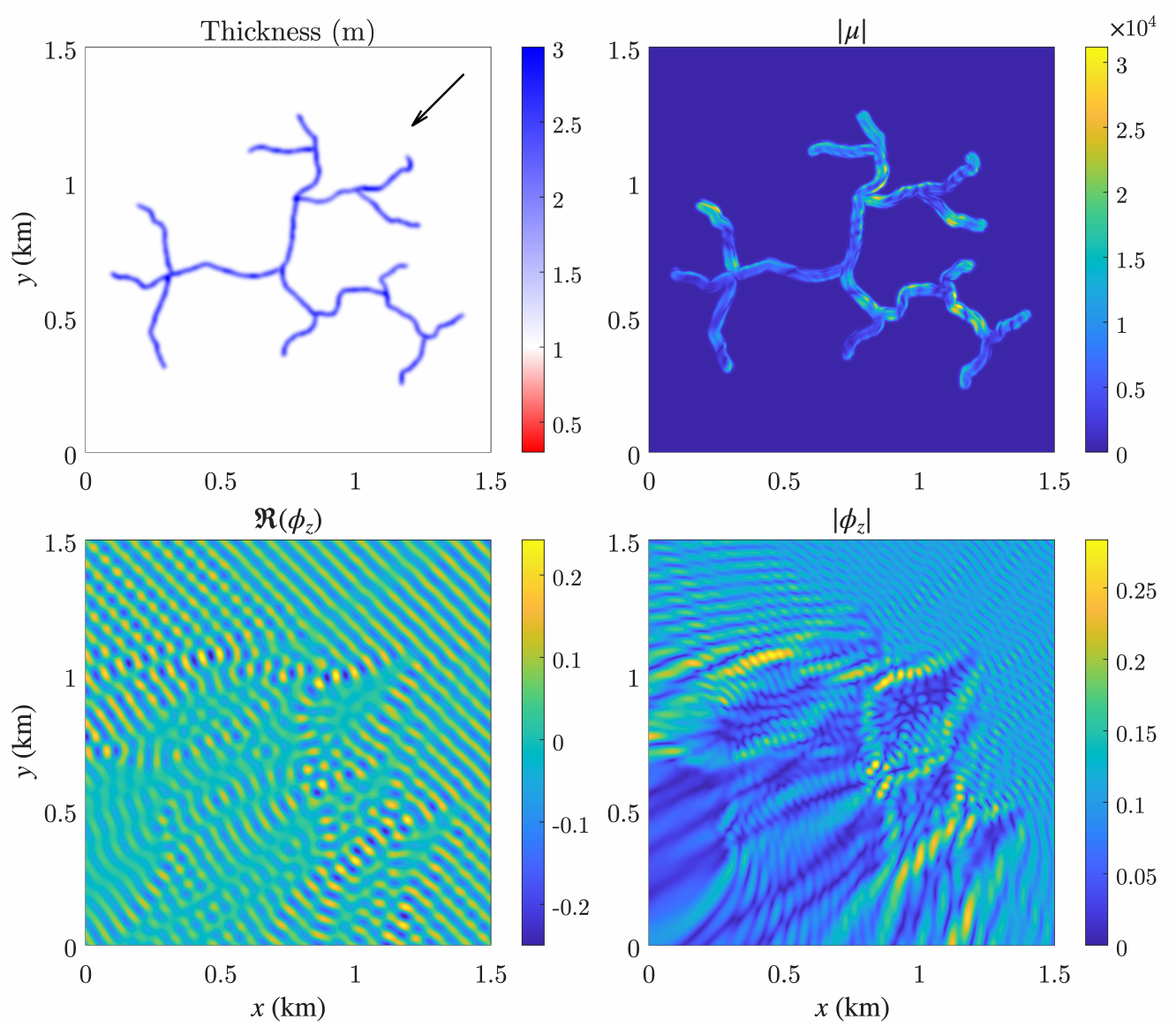}
    \caption{Wave scattering by a system of ridges. Top left: thickness profile. Top right: magnitude of the solution to the integral equation. Bottom left: Real part of the vertical velocity. Bottom right: magnitude of the vertical velocity. The black arrow indicates the direction of incident field. }
    \label{fig:ridges5}
\end{figure}

\begin{figure}[h]
    \centering
    \includegraphics[width=\linewidth]{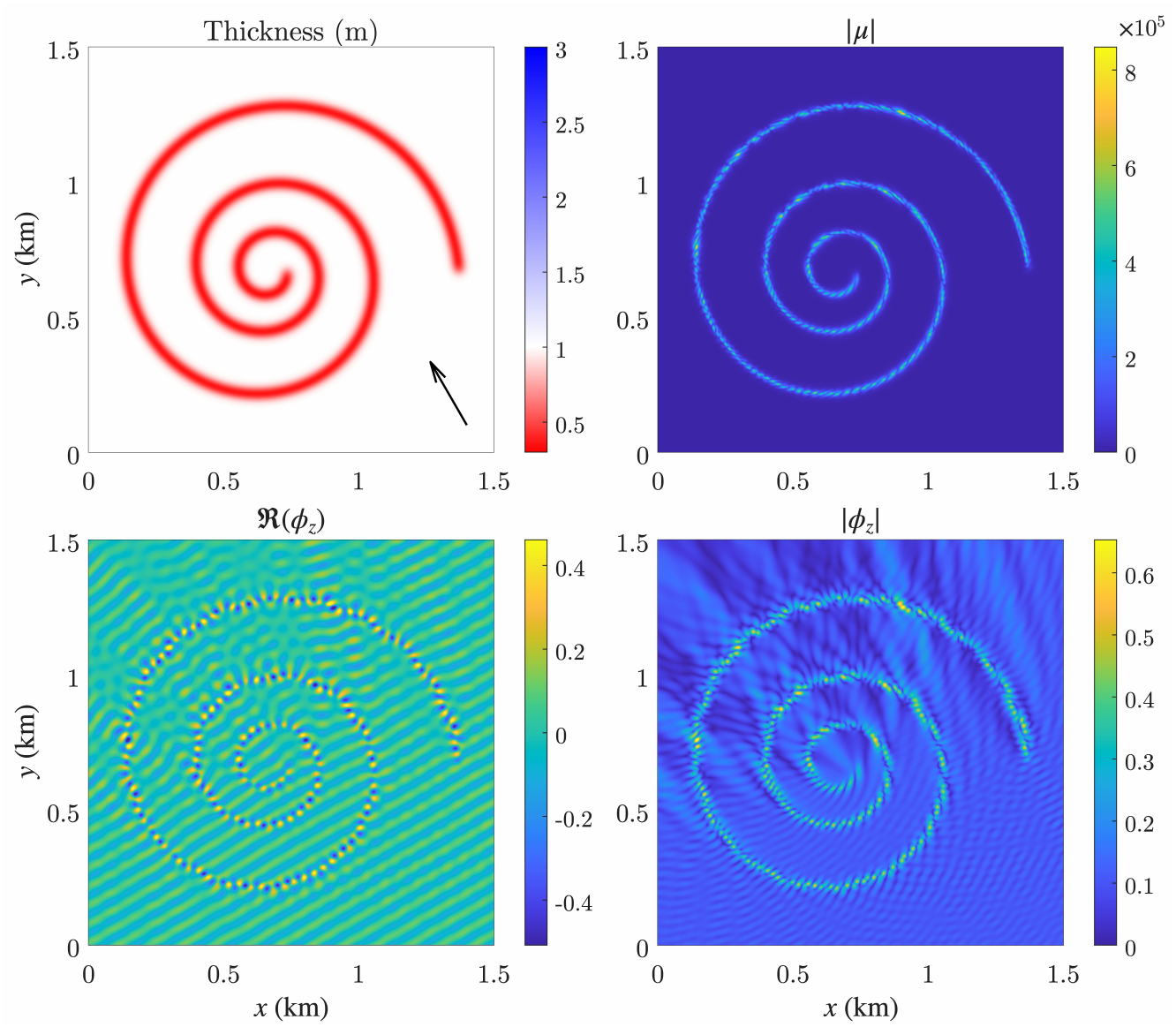}
    \caption{A spiral waveguide, with the same wavenumber as in the previous example (\cref{fig:ridges5}). Top left: thickness profile. Top right: magnitude of the solution to the integral equation. Bottom left: Real part of the vertical velocity. Bottom right: magnitude of the vertical velocity. The black arrow indicates the direction of incident field. }
    \label{fig:spiral}
\end{figure}

In the next example we look at wave interaction with a spiral-shaped channel or groove in the ice (\cref{fig:spiral}). This groove is parametrized using the following formula:
\begin{align*}
    \br_s(t) = ( -a t^3 \cos(t) , a t^3 \sin(t)) \; ,
\end{align*}
where $a = 0.017$. Then, the thickness $H$ is given by the following convolution of the speed of the curve with the Gaussian kernel:
$$ H(\br) = H_0 + A\int^{11\pi}_{5\pi} \exp\left( \frac{ -| \br - \br_s(t) |^2}{2\sigma^2}\right) | \br_s'(t)| \, \dd t \; , $$
where $\ H_0 = 1, \ \sigma = 17.5,$ and $A$ is chosen so that the minimum thickness of the ice is 0.35. This integral is performed using standard Gauss-Legendre quadrature, and derivatives of $H$ are readily obtained by taking derivatives of the Gaussian kernel above. The incident plane wave is chosen with wavenumber to be the same as the previous example, $k = 0.105 \ {\rm m}^{-1}${\hide{red}, and the solution converges in 1044 iterations of GMRES}. When the incident wave reaches the groove, it excites a trapped mode that resonates throughout the spiral. These types of
``waveguide'' phenomena for surface waves are a rich area for future numerical investigations.

\clearpage

\section{Conclusion}\label{sec:conclusion}

In this work, we develop a novel integral equation representation
for the variable thickness problem in the modeling of flexural-gravity waves. The resulting integral equation is second-kind and its solution is compactly supported on the region of varying thickness. We prove well-posedness of the integral equation, i.e., 
that solutions exist for data in appropriate function spaces and that the solutions are unique and depend continuously on the data. 
To solve the integral equation numerically, we apply high-order accurate quadrature and solve
the resulting linear system using an iterative algorithm with an FFT-accelerated matrix-vector product. 
The effectiveness and scalability of the method is demonstrated by solving several geometrically complex
flexural-gravity scattering problems, inspired by applications in glaciology 
that feature detailed branching patterns, Bragg scattering, multiple scattering, 
and trapped modes.

There are several extensions of this work that are of natural interest to applications in
glaciology. These include polynyas (holes in the ice), ice floes of finite extent,
and shore-fast ice. In these cases, the equations
will need to be augmented with appropriate boundary conditions in some instances, and new 
Green's functions developed in others.
Additionally, we have required that the coefficients $\alpha,\beta,$ and $\gamma$ are smooth but this assumption can be weakened. In particular, one can augment the system with additional constraints which enforce continuity of the solution and its derivatives across lines where $\alpha$ is not differentiable. For a detailed discussion on how this is done in two dimensions, see \cite{williams2004oblique}. We also note that the present method should extend easily to the case in which the depth of the fluid is finite but constant, via a standard \emph{method of images} type approach. Though the form of the Green's function will change, much of the analysis, discretization, and solution should remain the same, \emph{mutatis mutandis}. Such methods may be important for seismic imaging and passive sea-ice monitoring \cite{shuchman1994extraction}.

In this paper, a simple FFT-accelerated
iterative solver was used. However, for more singular problems, e.g. problems in which 
the contrast ratio of $\alpha$ is high, many iterations may be required. 
For such nearly-singular problems,
modern {\em fast direct} solvers can be easily adapted and extended to solve the discretized system. Such methods have the added benefit that, after an initial factorization step, it is relatively cheap to solve the system for many right-hand sides, which is particularly useful in optimization and inverse problems.

\section{Code availability}

 {\hide{red} The code used to generate the examples from this paper may be found at \url{https://doi.org/10.5281/zenodo.22117137}. }

\section{Acknowledgments}
The authors would like to thank Douglas MacAyeal, Mary Silber, Anand Oza, Bowei Wu, Tristan Goodwill, and Michael Siegel for many useful discussions. P. Nekrasov would like to acknowledge the support of the National Science Foundation (NSF) under Grant No. 2332479. J.G. Hoskins and M. Rachh would like to thank the American Institute of Mathematics and, in particular, John Fry for hosting them on Bock Cay during the SQuaREs program, where parts of this work were completed. The Flatiron Institute is a division of the Simons Foundation. 

\frenchspacing
\bibliographystyle{plain}
\bibliography{refs}

\appendix

\input{supplementary}

\end{document}

%% file: supplementary.tex
\section{Derivation of the Green's function}
\label{app:greensfun}

By the translation invariance of \eqref{eq:GSid}, we can compute
$\GS(\br,\br') = \GS(\br-\br',0)$. Let $\hat{\sigma}(\xi)$ be the
Fourier transform of $\GS(\br,0)$. Taking a Fourier transform of
\eqref{eq:GSid}, we obtain
$$ \left[(\alphn |\xi|^4-\betan) +\frac{\gamma}{|\xi|}\right] \hat{\sigma}(\xi) = 2. $$
Thus,
\begin{align}\label{eqn:hatsig} \hat{\sigma}(\xi) = \frac{2|\xi|}{\alphn |\xi|^5-\betan |\xi|+\gamma}.
\end{align}

\subsection{No repeated roots}

Let $\rho_1,\ldots,\rho_5$ be the roots of $p(z)=\alphn z^5-\betan z +\gamma$
and let $e_1,\ldots,e_5$ be defined as in \eqref{eq:moments}. For now, we will
assume that the roots are distinct as assumed in
Lemma~\ref{lem:dispersionpoly} and Theorem~\ref{thm:greensfun}. We discuss
appropriate modifications for treating repeated roots in Appendix~\ref{sec:repeatedroots}.

We perform a partial fraction expansion, obtaining
$$\hat{\sigma}(\xi)= 2 |\xi| \sum_{j=1}^5 \frac{e_j}{|\xi|-\rho_j} \; .$$
Recalling the moment relations \eqref{eq:momentrelation},
$\hat{\sigma}$ can be re-expressed as 

$$ \hat{\sigma}(\xi) =  2 \sum_{j=1}^5 \frac{e_j \rho_j}{|\xi|-\rho_j}.$$

Taking an inverse Fourier transform, we obtain
$$\GS(\br,0) = \frac{1}{\pi} \sum_{j=1}^5 e_j \rho_j
\int_0^\infty \rho J_0(\rho r) \frac{1}{\rho - \rho_j}\,{\rm d}\rho \; ,$$
where $r = |\br|$. Rearranging, we see that
$$\GS(\br,0)= \frac{1}{\pi}
\sum_{j=1}^5 e_j \rho_j \int_0^\infty J_0(\rho r) \,{\rm d}\rho
+ \frac{1}{\pi} \sum_{j=1}^5 e_j \rho_j^2 \int_0^\infty J_0(\rho r)
\frac{1}{\rho - \rho_j}\,{\rm d}\rho.$$
The moment relations~\eqref{eq:momentrelation} imply that
the first term is zero.

Let $\chi \in \bbC$ be given with $\arg(\chi) \ne 0$.
From \cite[\S 6.562]{gradshteyn2014table} we have the formula
$$ \int_0^\infty J_0(\rho r) \frac{1}{\rho - \chi}\,{\rm d}\rho
= \frac{\pi}{2} {\mathbf K}_0(-\chi r)  \; , $$
where $\mathbf{K}_0$ is the standard Struve function. This formula covers
most of the terms. 

By Lemma~\ref{lem:dispersionpoly}, $\rho_1$ is the only root with zero
argument. In this case, we determine the appropriate value for the
limit by the limiting absorption principle. If $\rho_1(\beta)$
is the root of $\alpha z^5 - \beta z + \gamma$ continued
from the root $\rho_1$ corresponding to $\beta = \beta_0$, then
it is straightforward to show that 
$$ \left. \frac{\partial \rho_1(\beta)}{\partial \beta} \right
|_{\beta=\beta_0} = \frac{\rho_1}{5\alpha \rho_1 - \beta_0}
=  \frac{\rho_1}{p'(\rho_1)}> 0 \; .$$
Thus, for $\beta = \beta_0+i\epsilon$, we have $\rho_1(\beta) \to
\rho_1$ approaches from the upper half plane as $\epsilon \to 0^+$.

The function $\mathbf{K}_0$ has a branch cut on the negative
real axis. In particular, $\mathbf{K}_0(z) = \mathbf{H}_0(z) - Y_0(z)$,
where $\mathbf{H}_0$ is an analytic Struve function and $Y_0$
captures the branch cut. By convention, the value of $Y_0$ on the
negative real axis is the limit of the principle branch from the
upper half plane.
We have also that $Y_0(ze^{m\pi i}) = Y_0(z) + 2im J_0(z)$.
After some re-arranging and applying standard identities, we obtain

$$ \lim_{\epsilon \to 0^+} \mathbf{K}_0(-\rho_1(\beta_0+i\epsilon)r)
= -\mathbf{K}_0(\rho_1 r) + 2i H_0^{(1)}(\rho_1 r) \; .$$
This recovers the formula \eqref{eq:GS}.

Let $\hat{\sigma_\phi}$ be the Fourier transform of $\Gphi(\br,0)$.
By the definition of $\Gphi$, \eqref{eq:Gphiid}, we have that

$$ \hat{\sigma_\phi}(\xi) = \frac{\hat{\sigma}(\xi)}{2|\xi|} =
\frac{1}{\alphn |\xi|^5-\betan |\xi|+\gamma} =
\sum_{j=1}^5 \frac{e_j}{|\xi|-\rho_j} \; .$$
Taking an inverse Fourier transform, we obtain
\begin{multline*}
  \Gphi(\br,0) = \frac{1}{2\pi} \sum_{j=1}^5 e_j 
\int_0^\infty \rho J_0(\rho r) \frac{1}{\rho - \rho_j}\,{\rm d}\rho \\
= \frac{1}{2\pi}
\sum_{j=1}^5 e_j \int_0^\infty J_0(\rho r) \,{\rm d}\rho
+ \frac{1}{2\pi} \sum_{j=1}^5 e_j \rho_j \int_0^\infty J_0(\rho r)
\frac{1}{\rho - \rho_j}\,{\rm d}\rho.
\end{multline*}
Again, the moment relations and the Struve function integral identity
applied above result in the formula { \eqref{eq:Gphieps} and the 
$\epsilon \to 0^+$ limit recovers \eqref{eq:Gphi}.}

\begin{remark}
The integral that defines $G_\phi$ in terms of $\GS$, i.e. \cref{eq:Gphiid},
is absolutely convergent for $\epsilon\ne 0$ but only conditionally convergent 
for $\epsilon=0$; c.f. \cref{prop:gsrhodecay}. Observe that the 
limiting absorption approach is consistent with the principal value definition
of the integral, as in \cref{eq:Sdef}, i.e. the formula \cref{eq:Gphiid} relating
$G_\phi$ and $\GS$ indeed holds. 
\end{remark}

\begin{remark}
The above treats the physically meaningful case, $\gamma < 0$. 
  For the non-physical case, $\gamma > 0$, the polynomial 
  $p(z) = \alphn z^5 - \betan z + \gamma$ can have two positive roots
  for real $\betan$.
  Because the derivative is positive at the larger of these roots
  and negative at the other, the dependence of the roots on $\betan$
  takes opposite sign. One then finds that, when applying a limiting absorption
  principle, the roots approach the real axis from opposite sides. This
  has the amusing effect that there are two Hankel function terms in the limiting
  Green's function, one that is outgoing and one that is incoming. 
\end{remark}

\subsection{Repeated roots}

\label{sec:repeatedroots}
In the case that $p(z) = \alphn z^5 -\betan z + \gamma$ has repeated
roots, the procedure can be modified appropriately. Suppose that
the roots are $\rho_1,\ldots,\rho_5$ where $\rho_4=\rho_5$ and
$\rho_1,\ldots,{ \rho_4}$ are distinct. Then, we have a partial fraction
decomposition

$$ \frac{1}{\alphn z^5 - \betan z + \gamma} = \sum_{j=1}^4 \frac{e_j}{z-\rho_j}
+ \frac{\tilde{e}_4}{(z-\rho_4)^2} \; .$$
The coefficients $e_1,\ldots,e_4$ are defined by residues as
before, i.e. $e_j = \Res_{z=\rho_j} 1/p(z)$ for $j=1,\ldots,4$.
The remaining coefficient can be found by the formula

$$\tilde{e}_4 = \Res_{z=\rho_4} \frac{z-\rho_4}{p(z)} \; .$$

Similar moment relations can also be found by the residue calculus.
We have

$$
\sum_{j=1}^4 e_j \rho_j^q + q \tilde{e}_4 \rho_4^{q-1} = 0
\textrm{ for } q\in\{0,1,2,3 \} \, , \
\sum_{j=1}^4 e_j \rho_j^4 + 4 \tilde{e}_4 \rho_4^{3} = \frac{1}{\alphn} \, .$$
We also have

$$ \frac{2z}{\alphn z^5 - \betan z + \gamma} = 2\sum_{j=1}^4
\frac{e_j\rho_j}{z-\rho_j} + 2 \frac{\tilde{e}_4}{z-\rho_4} 
+ 2 \frac{\tilde{e}_4 \rho_4}{(z-\rho_4)^2} \; . $$
  
A similar construction to the case with distinct roots gives 

\begin{multline*}
\GS(\br,0) = \frac{1}{\pi} \sum_{j=1}^4 e_j \rho_j
\int_0^\infty \rho J_0(\rho r) \frac{1}{\rho - \rho_j}\,{\rm d}\rho
+ \frac{1}{\pi} \tilde{e}_4 \int_0^\infty \rho J_0(\rho r) \frac{1}{\rho - \rho_4}\,{\rm d}\rho \\
  + \frac{1}{\pi}  \tilde{e}_4 \rho_4
  \int_0^\infty \rho J_0(\rho r) \frac{1}{(\rho - \rho_4)^2}\,{\rm d}\rho\; ,
\end{multline*}
  Rearranging, we see that
  \begin{multline*}
    \GS(\br,0)= \frac{1}{\pi}
    \left (\sum_{j=1}^4 e_j  \rho_j + \tilde{e}_4 \right )
    \int_0^\infty J_0(\rho r) \,{\rm d}\rho
    + \frac{1}{\pi} \sum_{j=1}^3 e_j \rho_j^2
    \int_0^\infty J_0(\rho r)
    \frac{1}{\rho - \rho_j}\,{\rm d}\rho \\
    + \frac{1}{\pi}
    \left (e_4 \rho_4^2 + 2 \tilde{e}_4 \rho_4 \right )  \int_0^\infty
    J_0(\rho r)\frac{1}{\rho - \rho_4} \,{\rm d}\rho
    + \frac{1}{\pi} \tilde{e}_4 \rho_4^2 \int_0^\infty J_0(\rho r)
    \frac{1}{(\rho - \rho_4)^2}\,{\rm d}\rho \; .
  \end{multline*}
  Applying the same identities as before, we obtain

\begin{multline*}
  \GS(\br,0) = \frac{1}{2} e_1 \rho_1^{2} \left  [
    - \mathbf{K}_0(\rho_1 r) + 2 i H_0^{(1)} (\rho_1 r)
    \right ]   + \frac{1}{2} \sum_{j=2}^3 e_j \rho_j^2
  \bK_0(-\rho_j r) \\
  + \frac{1}{2} \left ( e_4 \rho_4^2 + 2\tilde{e}_4 \rho_4 \right)
  \bK_0(-\rho_4 r)
  + \frac{1}{2} \tilde{e}_4 \rho_4^2 \left. \frac{\partial}{\partial \rho}
  \bK_0(-\rho r) \right |_{\rho = \rho_4} \; ,
\end{multline*}
or more explicitly

\begin{multline*}
  \GS(\br,0) = \frac{1}{2} e_1 \rho_1^{2} \left  [
    - \mathbf{K}_0(\rho_1 r) + 2 i H_0^{(1)} (\rho_1 r)
    \right ]   + \frac{1}{2} \sum_{j=2}^3 e_j \rho_j^2
  \bK_0(-\rho_j r) \\
  + \frac{1}{2} \left ( e_4 \rho_4^2 + 2\tilde{e}_4 \rho_4 \right)
  \bK_0(-\rho_4 r)
  + \frac{1}{2} \tilde{e}_4 \rho_4^2 \left ( -\frac{2}{\pi} + \bK_1(-\rho_4 r)
  \right ) r  \; .
\end{multline*}

\subsection{Analyticity of $\GS$ and its spatial derivatives}
We conclude with the following results which characterize the analyticity of $\GS$ in the parameter $\beta$,
independent of the presence of multiple roots.

{
\begin{proposition}\label{prop:positive_im}
   Fix $\alphn>0$ and $\gamma<0$ and let
   $$Q(\br):= \frac{1}{2 \pi \alpha_0} |\br| K_1(|\br|),$$
   where $K_1$ denotes the modified Bessel function of the second kind. Then, for any $\epsilon >0,$ the Green's function $\GS$ defined in \cref{eq:GSid} satisfies $\GS-Q \in H^{5-\epsilon}(\mathbb{R}^2),$ the map $\betan \mapsto \GS - Q$ is analytic on $\{\Im\betan > 0\}$ as an
    $H^{5-\epsilon}(\mathbb{R}^2)$-valued function, and
    \begin{equation*}
        \GS(\br) - Q(\br) = O(1/|\br|^3)\, , \qquad |\br| \to \infty,
    \end{equation*}
    uniformly for $\betan$ in compact subsets of $\{\Im\betan > 0\}$.
\end{proposition}
\begin{proof}
{Writing $z=|\xi|$, we have $\widehat{\GS} = 2z/p(z)$ and $\widehat{Q} =
\frac{2}{\alphn}(1+z^2)^{-2}$, and by \cref{lem:dispersionpoly} the
denominator satisfies $|p(z)| \ge C(\betan)(1+z^5)$ on $[0,\infty)$, for some constant $C(\betan)>0$ continuous in $\betan$. Since $\widehat{\GS} = \frac{2}{\alphn}z^{-4} +
O(z^{-8})$ while $\widehat{Q} = \frac{2}{\alphn}z^{-4}(1-2z^{-2}+O(z^{-4}))$,
\begin{equation*}
    \widehat{\GS}-\widehat{Q} = \frac{4}{\alphn}z^{-6} + O(z^{-8})\,,
\end{equation*}
and the difference is bounded near the origin; the first two claims follow.
Near $\xi=0$,
$\widehat{\GS}-\widehat{Q} = c_0 + \frac{2}{\gamma}|\xi| + O(|\xi|^2)$, and
the $|\xi|$ term gives the $O(1/|\br|^3)$ decay.}    
\end{proof}
}{
\begin{proposition}\label{prop:analyticbeta}
For fixed $\alphn>0$ and $\gamma<0$, the Green's function $\GS$ defined in
\cref{eq:GSid} extends analytically in $\betan$ from $\Im\betan>0$ to an open
neighborhood of $\{\Im\betan\ge 0\}$, on which
\begin{equation}\label{eq:GScontdecomp}
    \GS(\br) = i e_1\rho_1^2 H_0^{(1)}(\rho_1|\br|)
    - \frac{2}{\pi}e_1\rho_1^2\,K_0(\rho_1|\br|)
    - \frac{\gamma e_1}{2\pi\alphn\rho_1}\,|\br|K_1(|\br|)
    + R(\br,\betan)\,,
\end{equation}
with $R(\cdot,\betan)\in H^{5-\epsilon}(\mathbb{R}^2)\subset
C^3(\mathbb{R}^2)$ depending analytically on $\betan$ and satisfying $R =
O(1/|\br|^3)$, locally uniformly in $\betan$. For real $\betan$ the extension
coincides with the limiting absorption solution.
\end{proposition}
\begin{proof}
Fix $\beta_\ast\in\mathbb{R}$ and set
\begin{equation*}
    \hat{T}_1(z) = 4e_1\rho_1^2\left[\frac{1}{z^2-\rho_1^2} -
    \frac{1}{z^2+\rho_1^2}\right] = \frac{8e_1\rho_1^4}{z^4-\rho_1^4}\,,
    \qquad \hat{T}_2 = \widehat{\GS}-\hat{T}_1\,.
\end{equation*}
The residue of $\widehat{\GS}$ at its simple pole $z=\rho_1$ is $2e_1\rho_1$,
the same as that of $\hat{T}_1$, and neither function has any other pole near
$[0,\infty)$; hence $\hat{T}_2$ is analytic in a neighborhood of $[0,\infty)$,
jointly in $\betan$ near $\beta_\ast$. At infinity the Laurent expansions of
$\widehat{\GS}$ and $\hat{T}_1$ contain only the powers $z^{-4}, z^{-8},
z^{-9},\dots$, so
\begin{equation*}
    \hat{T}_2 = a_0 z^{-4} + O(z^{-8})\,, \qquad
    a_0 = \frac{2}{\alphn} - 8e_1\rho_1^4 =
    -\frac{2\gamma e_1}{\alphn\rho_1}\,,
\end{equation*}
where we have used the fact that $p'(\rho_1) = 4\alphn\rho_1^4-\gamma/\rho_1$. Subtracting
$a_0(1+z^2)^{-2},$ the transform of $\frac{a_0}{4\pi}|\br|K_1(|\br|)$, from $\hat{T}_2$
leaves a remainder which is $O(z^{-6})$ and bounded at the origin. As in the previous proposition, taking the inverse Fourier transform of $\hat{T}_2 -a_0(1+z^2)^{-2}$ leaves a remainder $R$ which is in
$H^{5-\epsilon}(\mathbb{R}^2)$, is analytic in $\betan$, and is
$O(1/|\br|^3)$. Finally, for $\Im\rho_1>0,$ standard Fourier transform identities~\cite{abramowitz1948handbook} applied to
$(|\xi|^2\mp\rho_1^2)^{-1}$ give
\begin{equation*}
    T_1(\br) = ie_1\rho_1^2 H_0^{(1)}(\rho_1|\br|) -
    \frac{2}{\pi}e_1\rho_1^2\,K_0(\rho_1|\br|)\,,
\end{equation*}
which is analytic in $\rho_1$ off the closed negative real axis. Thus, it 
defines the continuation of $T_1$ to all $\betan$ near $\beta_\ast,$ which proves \cref{eq:GScontdecomp}. For
$\Im\betan<0$ the continuation of $T_1$ is unbounded in $\br,$ with $|T_1|\sim C
e^{\max\{0,-\Im\rho_1\}|\br|}/\sqrt{|\br|}.$
Finally, we note that $|\br|K_1(|\br|)$ is independent of $\beta_0,$ behaves like $1 + |\br|^2/2 \log (|\br|/2) +O(|\br|^3)$ for small $|\br|,$ and decays exponentially as $|\br| \to \infty$. 
\end{proof}

\begin{remark}
The $|\br|^2\log|\br|$ singularity of $\GS$ is carried entirely by the
explicit terms in \cref{eq:GScontdecomp}: their contributions,
$\frac{e_1\rho_1^4}{\pi}$ and $-\frac{\gamma e_1}{4\pi\alphn\rho_1}$, sum to
$\frac{1}{4\pi\alphn}$ (independent of $\beta_0$), in agreement with \cref{prop:positive_im}.
\end{remark}

\begin{corollary}\label{cor:opanalytic}
The kernels \eqref{kernel1}-\eqref{kernel8} extend analytically in $\betan$ to this
neighborhood, and the corresponding integral operators form an analytic
family of compact operators on $L^2(\Omega)$.
\end{corollary}
\begin{proof}
Differentiating \cref{eq:GScontdecomp} up to three times expresses each
kernel as an explicit weakly singular part with analytic coefficients plus a
remainder $\partial^k R \in H^{2-\epsilon}(\mathbb{R}^2)\subset
C^0(\mathbb{R}^2)$, also analytic in $\betan.$ Both yield compact operators
on the bounded domain $\Omega.$ Note that the growth of $T_1$ at infinity does not affect the compactness, since $\Omega$ is bounded. The same argument applies to $\Gphi$, whose transform
$1/p(|\xi|)$ admits the analogous decomposition with oscillatory part
$\frac{i}{2}e_1\rho_1 H_0^{(1)}(\rho_1|\br|)$.
\end{proof}}

\section{Numerical evaluation of Struve and related functions}\label{app:evaluation}
Though Struve functions arise in a wide variety of applications, there appear to be relatively few accurate and efficient algorithms for their evaluation. In this section we briefly describe a numerical method for their computation. For ease of exposition, we define the related functions ${\bf R}_n$ defined by
$${\bf R}_n(z) := J_n(z) + i {\bf H}_n(z).$$
Here we focus on ${\bf R}_0$ and ${\bf R}_1.$ In principle, together with the three-term recurrence,
$${\bf R}_{\nu-1}(z)+{\bf R}_{\nu+1}(z) = \frac{2\nu}{z}{\bf R}_{\nu}(z)+ \frac{i \left(\frac{z}{2}\right)^{\nu}}{\sqrt{\pi}\Gamma\left(\nu + \frac{3}{2} \right)} \; , $$
this can be extended to the computation of higher-order Struve functions. Derivatives of ${\bf R}_\nu$ can similarly be calculated using the relation
$${\bf R}_{\nu-1}(z) - {\bf R}_{\nu+1}(z) = 2 {\bf R}_{\nu}'(z)- \frac{i \left(\frac{z}{2} \right)^{\nu}}{\sqrt{\pi}\Gamma\left(\nu+\frac{3}{2} \right)} \; .$$
We consider two separate cases, depending on the magnitude of $z.$ For small $z$ we use generalized Gaussian quadrature together with an integral representation of ${\bf R}_{n}$. For large $z$ we use an asymptotic expansion.

In the first case, we use the following Lemma which follows immediately from \cite{abramowitz1948handbook}.
\begin{lemma}
    Let ${\bf R}_n(z)$ be the function defined by ${\bf R}_n(z)= J_n(z) +i {\bf H}_n(z).$ Then, for any $z \in \mathbb{C}$, and $n\ge 0,$ ${\bf R}_n$ has the integral representation
    $${\bf R}_n(z) = \frac{2
\left( \frac{z}{2}\right)^n}{\sqrt{\pi}\Gamma\left(n+ \frac{1}{2}\right)} \int_0^1 (1-t^2)^{n-1/2}e^{izt}\,{\rm d}t.$$
\end{lemma}
Next, we have the following result which provides an upper bound on the number of quadrature nodes required to numerically compute the above integrals for $z$ in the upper halfplane with non-negative imaginary part.
\begin{proposition}
For $z \in \mathbb{C}$ with $\Im (z) \ge 0,$ consider the function $f_z:[0,1] \to \mathbb{C}$ defined by
$$f_z(t) = \frac{(1-t^2)^{n}}{\sqrt{1+t}}e^{izt}.$$
For any $\epsilon>0$ there exists an $m < C(n+|z|+\log \epsilon^{-1})$ together with a polynomial $p \in P_m$ such that $\|p - f_z\|_{\infty}<\epsilon.$ Moreover, if $|z|<R$ then $m$ can be chosen independent of $|z|$.
\end{proposition}
In light of the above proposition, given a tolerance $\epsilon,$ and letting $m = m(n,\epsilon)$, then it follows that there exists an $m$ point quadrature rule which integrates 
$$I_{n,z} =\int_0^1 (1-t^2)^{n-1/2}e^{izt}\,{\rm d}t \; , $$
with an error bounded by $2\epsilon$ for any $z \in \mathbb{C},$ with $|z|<R,$ and $\Im z \ge 0.$ 

The above analysis guarantees slow growth in the size of the quadrature rule as a function of the $z$ cutoff, accuracy requirements, and order. In principle, it can be used constructively to produce suitable quadratures. In our calculations we fix $\epsilon = 10^{-12}$ and $R=95$. Using a generalized Gaussian quadrature method \cite{ggq,ggq2}, we obtain a 38 point quadrature rule to approximate $I_{n,z}$ for $n=0,1.$
In order to be able to compute ${\bf R}_{0,1}$ everywhere in the upper halfplane we require an algorithm for computing values for $|z| \ge 95.$ Here we use the asymptotic series. The cut-off of $|z|= 95$ is chosen to guarantee that the asymptotic series gives at least an absolute error of $10^{-12}$. In particular, we have the expansion \cite{abramowitz1948handbook}
$${\bf K}_{\nu}(z) \sim \frac{1}{\pi}\sum_{n=0}^\infty \frac{\Gamma(n+1/2)\, \left( \frac{z}{2}\right)^{\nu-2{n}-1}}{\Gamma(\nu+1/2-{n})}.$$
This, coupled with the identity
{$${\bf R}_\nu(z) = H_\nu^{(1)}(z)+i {\bf K}_{\nu}(z),$$}
gives an asymptotic expansion for ${\bf R}_\nu$ in terms of the above expansion and the standard asymptotic expansions for the Hankel function $H_\nu^{(1)}$ \cite{abramowitz1948handbook}.

\section{Derivatives of the Green's function}\label{app:greensderivatives}

In this section, we compute the derivatives of the Green's function \eqref{eq:GS}. Recall that we have the following recurrence relations for the Hankel function of the first kind: 
\begin{align*}
    H^{(1)'}_0 (z) &= - H^{(1)}_1 (z) \, , \\
    H^{(1)'}_1 (z) &= H^{(1)}_{0} (z) - \frac{1}{z}  H^{(1)}_1 (z) \, ,
\end{align*}
As well as the following recurrence relations for the Struve function of the second kind:
\begin{align*}
    \mathbf{K}'_0 (z) &= \frac{2}{\pi} - \mathbf{K}_1 (z) \; , \\
    \mathbf{K}'_1 (z) &= \mathbf{K}_{0} (z) - \frac{1}{z} \mathbf{K}_1 (z) \; .
\end{align*}
Taking derivatives of \eqref{eq:GS} with respect to the target variable $\br$, we obtain:
\begin{align*}
      &\nabla_\br \, \GS(\br,\br')= \frac{\br - \br'}{|\br - \br'|} \, M_1(\br, \br') \; , \\
    & H [ \GS ](\br, \br')=  \frac{-1}{|\br-\br'|^3} \begin{pmatrix}
       (x-x')^2 - (y-y')^2 & 2(x-x')(y-y') \\
       2(x-x')(y-y') & (y-y')^2 - (x-x')^2 
   \end{pmatrix} M_1(\br,\br') \nonumber \\
   &\qquad\qquad\qquad\qquad\qquad + \frac{1}{|\br-\br'|^2} \begin{pmatrix}
       (x-x')^2 & (x-x')(y-y') \\
       (x-x')(y-y') & (y-y')^2
   \end{pmatrix} M_2(\br,\br') \; , \label{eq:hessGS} \\
   & \Delta_\br \,  \GS ( \br, \br') =  M_2(\br, \br') \; , \\
   & \nabla_\br \, \Delta_\br \, \GS(\br,\br') = - \frac{\br - \br'}{|\br - \br'|} \, M_3(\br,\br') \; ,
\end{align*}
where $H [\GS]$ is the Hessian of $\GS$ and the functions $M_1, M_2, M_3$ are given by:
\begin{align*}
    M_1(\br,\br') &= \frac12 e_1 \rho_1^3 [\mathbf{K}_1 (\rho_1 |\br - \br'|) - 2i H_1^{(1)} (\rho_1 |\br- \br'| ) ] + \frac12 \sum_{j=2}^5 e_j \rho_j^3 \mathbf{K}_1(-\rho_j |\br - \br'|) \; , \\
    M_2(\br,\br') &= \frac{1}{2} e_1 \rho_1^{4}  [ \mathbf{K}_0(\rho_1 |\br - \br'|)   - 2 i H_0^{(1)} (\rho_1 |\br - \br'|)] -  \frac{1}{2}\sum_{j=2}^5 e_j \rho_j^4 \bK_0(-\rho_j |\br - \br'|) \; , \\
    M_3(\br,\br') &= \frac12 e_1 \rho_1^5 [\mathbf{K}_1 (\rho_1 |\br - \br'|) - 2i H_1^{(1)} (\rho_1 |\br- \br'| ) ] + \frac12 \sum_{j=2}^5 e_j \rho_j^5 \mathbf{K}_1(-\rho_j |\br - \br'|) \; .
\end{align*}

\section{Proof of convergence}
\label{app:cc}
Given a second kind integral
equation with a weakly-singular integral kernel with smooth enough coefficients and a suitable quadrature
rule, it is natural to expect the resulting Nystr\"{o}m scheme to converge at the same order as the quadrature 
rule. In this section, we prove this for a class of integral equations which includes~\cref{eq:LS}, discretized using a corrected trapezoid rule. The proof follows standard lines, see~\cite{kress,nystrom}; we include it here for completeness. The crux is to show that, for fine enough grids, the discretized linear system is invertible, with a bound on the inverse independent of the grid size. This is
established in \cref{cor:fulldisc}.

Given a region $\Omega$ containing the support of $f$, let $\Omega_{0} = [-L,L]\times[-L,L]$ be such $\Omega \subset \Omega_{0}$ and that the boundary of $\Omega$ is a distance $d>0$ from the 
boundary of $\Omega_{0}$.
We consider the case in which we are given an integral operator $K:C(\Omega_0) \to C(\Omega_0)$, with kernel
$$k(\bx,\by) = \sum_{j=1}^m \beta_j(\bx) \, w_j(\bx-\by).$$
We further assume that the functions $\beta_j$ are smooth and compactly supported in $\Omega$, and that each $w_j$ satisfies estimates of the kind
$$|w(\bx)| \le \frac{w_*}{|\bx|}, \quad |\nabla w(\bx)| \le \frac{w_d}{|\bx|^2} \; , $$
for all $|\bx| \le {\rm diam}\, \Omega_0.$ Let $\beta_*:= \max_j \| \beta_j\|_\infty$ and $\beta_d := \max_j \| \,|\nabla\beta_j|\,\|_\infty.$

In the discretization of our problem we consider equispaced grids with spacing $h.$ In the following, we assume that $h = 2L/(2N+1)$ for some large enough integer $N>0$. Let $\bxi = {\bf i}h,$ where ${\bf i} =(i_1,i_2),$ $i_{1},i_{2} \in -N,-N+1,\ldots , N-1, N$ denote the locations of the discretization nodes. Given the plethora of indices in the following, we suppress the dependence of $h$ on $N$. 

For each $h$, each point $\bxi$ and each $j$, we construct a quadrature for $w_j$ which is of the form $\omega_{\bf i,\bf i'}^{j,h} = w_j(h({\bf i} - {\bf i'}))\,h^2$ if ${\bf i}-{\bf i'} \notin \mathcal{C}$ and $\omega_{\bf i,\bf i'}^{j,h} = \alpha^j_{\bf i - \bf i'}(h)$ if ${\bf i} - {\bf i'} \in \mathcal{C}.$ Here $\mathcal{C}$ is a fixed finite set (the stencil) independent of $h$ and $\alpha_{\bf i}^j(h)$ are quadrature corrections which may depend on $h$ but satisfy the bound $|\alpha_{\bf i}^j(h)| \le \alpha_* h$. 
We further require $(0,0) \in \mathcal{C}.$
Furthermore, we assume that the stencil and corrections are chosen so that for some integer $p>0,$ and any smooth function $\mu$ on $\Omega_{0}$ 
\begin{equation}
\label{eq:quadacc}
\sup_{\bx^h_{\bf i}\in \Omega}\left|\sum_{\bf i'} \omega_{\bf i,\bf i'}^{j,h} \mu(\bxip) - \int_{\Omega_0} w_j(\bxi - \by)\,\mu(\by)\,{\rm d}\by\right| < A h^p \; ,
\end{equation}
for some constant $A$ depending on $\mu.$

The discretization of $K$ is then given by 
$$\underline{K}^h[\mu](\bxi) = \sum_{\bf i'} \sum_{j=1}^m \beta_j(\bxi) \, \omega_{\bf i,i'}^{j,h} \mu(\bxip).$$
We can extend this to a continuous function on $\Omega_0$ by setting $\underline{K}^h[\mu](\bx)$ to be the continuous piecewise linear interpolant agreeing with $\underline{K}^h[\mu](\bxi)$ at each $\bxi$, and $0$ on $\partial \Omega_0.$ The interpolant is chosen to be linear on each triangle formed by splitting one of the grid squares in half along the bottom left to top right diagonal. For ease of exposition we set {$L_{\bf i}^h(\bx)$} to be the coefficient of $\mu(\bxi)$ in $\underline{K}^h[\mu](\bx)${, i.e. for $\bx\in \Omega_0$, $$ \underline{K}^h[\mu](\bx) = \sum_{\bf i} L_{\bf i}^h(\bx) \mu(\bx^h_{\bf i}) \, . $$} 
As noted above, the main analytical result we will establish is that the
finite-rank operators $\{\underline{K}^h\}$ { plus the identity} and the corresponding 
fully-discrete operators for such a quadrature rule are invertible 
for sufficiently small $h$. The aim of this appendix is to prove the following three results.
\begin{theorem}
\label{thm:semidiscinvertible}
    Suppose that $I+K$ is invertible on { $ C_0(\Omega_0)  := \{  f \in C(\Omega_0) \; | \;  \{ | f| > \varepsilon \} \subset \Omega \text{ for any } \varepsilon > 0 \} $}. 
    In the notation above, $(I+\underline{K}^h)^{-1}$ exists and 
    $\|(I+\underline{K}^h)^{-1}\|_{C_0(\Omega_0)}$ is bounded by
    a constant independent of $h$, for $h$ sufficiently small.
\end{theorem}
The following corollary is an immediate consequence of the previous theorem.
\begin{corollary}
\label{cor:fulldisc}
    Under the same conditions as \cref{thm:semidiscinvertible}, 
    the matrix $\mathbf{I}+\mathbf{K}^h$, where $K^h_{\mathbf{i}\mathbf{j}} = L_{\bf j}^h(\bxi)$,
    is invertible and the $\ell^\infty$ norm condition number of the family
    of matrices satisfies a uniform bound for $h$ sufficiently small.
\end{corollary}
\begin{theorem}
\label{thm:convergence}
Suppose $f \in C^\infty_{c}(\Omega)$, extended to $0$ in $\Omega_{0}\setminus\Omega$. Suppose further that $\mu$ solves $(I+K)\mu = f$, and that $\mu^{h}$ solves $(I+K^{h})\mu^{h} = f$.
Then there exists a constant $C$ such that
$$\max_{\bf i} | \mu(\bx^h_{\bf i})-\mu^h(\bx^h_{\bf i})| \le C h^p \; , $$
    for h sufficiently small.
\end{theorem}

{ Before proving \cref{thm:semidiscinvertible} we briefly outline the relationship between invertibility on $L^2(\Omega)$ and $C_0(\Omega_0)$. The following theorem follows directly from Theorem 2.29 in \cite{kress} together with the fact that the integral operators map $C_0(\Omega_0)$ into itself.

\begin{theorem}
\label{thm:compactC}
  The integral operators associated with the kernels
  \cref{kernel1,kernel2,kern3,kern4,kern5,kern6,kern7,kernel8}
  are compact from $C_0(\Omega_0) \to C_0(\Omega_0)$.
\end{theorem}

\begin{corollary}\label{label:SKIE_C}
      Equation \cref{eq:LS} is an {integral equation of the second kind} in $C_0(\Omega_0)$. 
\end{corollary}

\cref{thm:inteqexistunique} establishes the uniqueness of solutions in $L^2(\Omega)$ for the integral equation in the dissipative regime. Moreover, \cref{cor:continuous} establishes that if the right-hand side $f$ is in $ C_0(\Omega)$, then the solution $\mu \in C_0(\Omega) \subset L^2(\Omega)$, giving uniqueness of the integral equation on $C_0(\Omega)$. Both $f$ and $\mu$ can be extended from $\Omega$ to $\Omega_0$ by the zero function, so the uniqueness of solutions on $C_0(\Omega_0)$ follows directly. Using \cref{label:SKIE_C}, we conclude that the integral equation is invertible on $C_0(\Omega_0)$ whenever $\beta$ is in the dissipative regime.
}

We prove \cref{thm:semidiscinvertible} using a mild extension of the strategy employed in \cite[Ch. 10--12]{kress}. The idea of the proof
is to apply the following result about the pointwise convergence of collectively 
compact families of operators, which is a consequence of the
Banach-Steinhaus theorem or uniform boundedness principle.

\begin{lemma}{\cite[Thm. 10.12 and Cor. 10.11]{kress}}
\label{lem:kressmain}
    Let $X$ be a Banach space and let $A:X\to X$ be a compact
    operator such that $I+A$ is injective on $X$. 
    Suppose that $H>0$ and that $\{A^h\}$ for $0<h\leq H$
    is a collectively compact set of operators such that 
    $A^h \mu \to A\mu$ as $h\to 0$ for all $\mu\in X$. Then,
    $(I+A^h)^{-1}$ exists and $\|(I+A^h)^{-1}\|_X$ is bounded
    by a constant independent of $h,$ for $h$ sufficiently small. 
\end{lemma}
To apply this result, we need to establish the compactness of
the original operator $K$, the collective compactness of the family 
of finite-rank operators $\{ \underline{K}^h\}$, and the pointwise
convergence of these same operators. That $K$ is compact follows immediately from standard results, see { \cite[Thm. 2.29]{kress}} for example, together with the behavior of the kernels near the diagonal.


In order to establish the collective compactness of the family $\{ \underline{K}^h\}$, we require the following two lemmas, whose proof we defer to the end of the appendix.

\begin{lemma}
\label{lem:khbounded}
    With $L_{\bf i}^h$ as defined above, if $h<d$ then there exists a constant $C_1$, independent of $h,$ such that
    $$\sum_{\bf i} |L_{\bf i}^h (\bx)| \le C_1 \; , $$
    for all $\bx \in \Omega_0.$ In particular,
    $$\sup_{0<h<d} \|\underline{K}^h\|_\infty < C_1.$$
\end{lemma}
\begin{lemma}
\label{lem:khequicontinuous}
    With $L_{\bf i}^h$ defined as above, then for each $\eta \in (0,1),$ there exists a constant $C_2$, depending on $\eta$ but not on $h$, such that
    $$\sum_{\bf i} |L_{\bf i}^h(\bx) - L_{\bf i}^h(\by)| \le C_2 \|\bx-\by\|^{\eta} \; , $$
    for all $0<h<{\min}\{d,1\}$ and all $\bx,\by\in \Omega_0.$ In particular, for any 
    bounded set $U\subset C(\Omega_0)$, the set 
    $$ \{\underline{K}^h \mu : \mu \in U \textrm{ and } 0<h< { \min\{d,1\}} \} \, , $$
    is equicontinuous.
\end{lemma}
From the two previous lemmas, the collective compactness follows immediately.
\begin{lemma}
    \label{lem:khcc}
    In the notation above, the operators $\{\underline{K}^h\}$, for $0< h < { \min\{d,1\}}$,
    are collectively compact. 
\end{lemma}

\begin{proof}
By the Arzel\`{a}-Ascoli theorem, collective compactness
is equivalent to the boundedness 
and equicontinuity of sets of the form $\{\underline{K}^h \mu \}$
where the $\mu$ come from a bounded subset of $C(\Omega_0)$ and
$0<h<{ \min\{d,1\}}$. These are established in
\cref{lem:khbounded,lem:khequicontinuous}, respectively. 
\end{proof}

{ We also require the following lemma, which gives the pointwise convergence of the operators $\underline{K}^h$ to $K.$
\begin{lemma}
\label{lem:khpw}
    For any ${\mu}\in C_0(\Omega_0),$ 
    $$\lim_{h\to 0^+}\sup_{\bx\in\Omega_0} \left| \underline{K}^h[\mu](\bx)-K[\mu](\bx)\right| = 0.$$
\end{lemma}
\begin{proof}
Consider a family of smooth mollifiers, $\varphi_\epsilon$. Because $\mu$ is
uniformly continuous, the functions $\mu_\epsilon = \varphi_\epsilon\star \mu$
have $\sup_\bx |\mu_\epsilon(\bx) - \mu(\bx)| \to 0$ as $\epsilon \to 0$. For each $\mu_\epsilon$ with $\epsilon > 0$, 
we may apply \cref{eq:quadacc}, properties of piecewise linear interpolants, and the fact that both $\underline K^{h}[\mu]$ and $K[\mu]$ are {identically} $0$ for $\bx \in \Omega_{0} \setminus \Omega$
to obtain 
$$ \sup_{\bx \in \Omega_0} \left | \underline{K}^h[\mu_\epsilon](\bx)-K[\mu_\epsilon](\bx)\right| \leq C(\epsilon) (h^{p} + h^2) \; .$$
Because $\underline{K}^h$ and $K$ have a uniform operator bound on bounded functions, we
have 
\begin{multline*}
\left| \underline{K}^h[\mu](\bx)-K[\mu](\bx)\right| \\= \left | \underline{K}^h[\mu_\epsilon](\bx) - K[\mu_\epsilon](\bx) 
+ \underline{K}^h[\mu-\mu_\epsilon](\bx) - K[\mu-\mu_\epsilon](\bx)\right | \\ \leq 
C(\epsilon)h^p + C_2 \sup_{\by} \|\mu(\by)-\mu_\epsilon(\by)\| \; ,
\end{multline*}
for any $\bx$. By selecting $\epsilon$ and $h$ sufficiently small, the above can be
made arbitrarily small, as desired. 
\end{proof}}

We now proceed with the proof of \cref{thm:semidiscinvertible}. 

\begin{proof}[Proof of \cref{thm:semidiscinvertible}, and \cref{cor:fulldisc}]
    Applying \cref{lem:khcc,lem:khpw}, we have that 
    the family $\{\underline{K}^h\}$ and the operator $K$ {satisfy}
    the conditions of \cref{lem:kressmain} on the space $C(\Omega_0)$,
    establishing \cref{thm:semidiscinvertible}.

    To prove \cref{cor:fulldisc}, we make the observation
    that the semi-discrete operators $I+\underline{K}^h$ map 
    piecewise linear interpolants over the grid points to
    piecewise linear interpolants over the grid points.
    Let $h$ be sufficiently small that the operators $I+\underline{K}^h$
    are invertible and the norms $\|(I+\underline{K}^h)^{-1}\|_{C(\Omega_0)}$ satisfy
    a uniform bound. 
    Suppose, by contradiction, that the discrete operator $\mathbf{I}+\mathbf{K}^h$
    has a non-zero null vector. Then, the piecewise linear interpolant of 
    the values specified by this null vector gets mapped to the zero 
    function by $I+\underline{K}^h$, a contradiction. Likewise,
    the correspondence of the $\ell^\infty$ norm of the interpolation
    values and the $C(\Omega_0)$ norm of the interpolant
    for piecewise linear interpolants over the grid points implies
    that $\| ({\bf{I}}+{\bf{K}}^h)^{-1}\|_\infty$ is bounded uniformly for
    sufficiently small $h$. 
\end{proof}
\begin{proof}[Proof of \cref{thm:convergence}]
    \Cref{cor:smoothness} along with the assumptions imply that $\mu \in C_{c}^{\infty}(\Omega_{0})$.  
    Applying the estimate~\cref{eq:quadacc} to $\mu$, we get that $\sup_{\bf i}|\underline{K}^{h}[\mu](\bx_{\bf i}) - K[\mu](\bx_{\bf i})| \leq C h^{p}$ for some constant $C$ and $h$ sufficiently small. Letting $\delta_{\bf i} = \mu(\bxi) - \mu^{h}(\bxi)$, we note that the difference satisfies the linear system 
$(\mathbf{I}+\mathbf{K}^h)\boldsymbol{\delta} = \boldsymbol{g}$, where
    $g_{\bf i} = (\underline{K}^{h} - K)[\mu](\bx_{\bf i})$. The result then follows from~\cref{cor:fulldisc}.
\end{proof}
We conclude with the proofs of the two lemmas,  \cref{lem:khbounded} and \cref{lem:khequicontinuous}. In the following we use $C$, and its various decorations, to denote arbitrary  constants independent of $h$, with values that can change from line to line. In our analysis we will make frequent use of the following bounds. The proof is straightforward and omitted.
\begin{lemma}\label{lem:app_bnd}
Let $p\in \mathbb{R}_+$ and $M\in \mathbb{Z}_+,$ be given. Then there exist constants $c>0$ and $C<\infty$, depending on $p$ but independent of $M$, such that
$$c\int_{1/2}^{M} \frac{1}{r^{p-1}}\,{\rm d}r\le \sum_{1\le\,\|{\bf i}\|_\infty\le M} \frac{1}{|{\bf i}|^p} \le C\int_{1/2}^{M} \frac{1}{r^{p-1}}\,{\rm d}r.$$
Moreover, there exists a constant $\bar{C}$ such that for any $1 \le \rho_1 < \rho_2<\infty,$
$$\sum_{\rho_1 \le |{\bf i}| \le \rho_2}\frac{1}{|{\bf i}|^p} \le \bar{C} \int_{\rho_1-\frac{1}{\sqrt{2}}}^{\rho_2+\frac{1}{\sqrt{2}}} \frac{1}{r^{p-1}}\,{\rm d}r. $$
\end{lemma}

\begin{proof}[Proof of \cref{lem:khbounded}]
    We first prove the inequality for $\bx = \bxi,$ the case of general $\bx$ following by linearity. Then,
    $$\sum_{{\bf i-i'} \notin \mathcal{C}} |L_{\bf i'}^h (\bx)| = \sum_{j=1}^m\sum_{{\bf i-i'} \notin \mathcal{C}} |\beta_j(\bxi)| \,|w_j(h({\bf i - \bf i'}))| h^2 \le \beta_* m\sum_{{\bf i-i'} \notin \mathcal{C}} w_* h^2 \frac{1}{h|{\bf i}-{\bf i}'|}. $$
    Using the previous lemma, it follows that
    \begin{align}\label{eq:app_bnd1}\sum_{{\bf i-i'} \notin \mathcal{C}} |L_{\bf i'}^h (\bx)| \le \tilde{C}\, {\rm diam}\, \Omega_0.
    \end{align}
    A bound on the remaining portion, the sum over ${\bf i-\bf i'} \in \mathcal{C}$, can be obtained from the bounds on the $\alpha_{\bf i}^h,$
    \begin{align}\label{eq:app_bnd2}\sum_{{\bf i -\bf i'} \in \mathcal{C}} |L_{\bf i'}^h(\bxi)| \le |\mathcal{C}| \beta_* m \alpha_* h.
    \end{align}
    Combining (\ref{eq:app_bnd1}) and (\ref{eq:app_bnd2}) yields the desired result.
\end{proof}

For ease of exposition, let $\mathcal{C}^{h}(\bx)$ denote the indices of the grid points with corrections for the point $\bx$ with grid size $h.$ We note that $|\mathcal{C}^h(\bx)| \le 4 |\mathcal{C}|.$ We now prove \cref{lem:khequicontinuous}.

\begin{proof}[Proof of \cref{lem:khequicontinuous}]
From the previous lemma, it suffices to establish this for $\|\bx-\by\| <1.$ We begin by observing that for any grid point $\bxl,$
$$L_{\bf i'}^h(\bxl) = \sum_{j=1}^m \beta_j(\bxl)w_j(h({\boldsymbol \ell}-{\bf i}'))h^2 + \sum_{j=1}^m \beta_j(\bxl)\tilde{\alpha}^j_{\boldsymbol\ell -\bf i'}(h)=: T_{\bf i'}^h(\bxl) + S_{\bf i'}^h(\bxl) \; , $$
where $\tilde{\alpha}^j_{\bf i} = \alpha^j_{\bf i} -w_j(h{\bf i})h^2$ if ${\bf i} \in \mathcal{C},$ ${\bf i}\neq (0,0)$, $\tilde{\alpha}^j_{(0,0)} = \alpha^j_{(0,0)},$ and $\tilde{\alpha}^j_{\bf i} = 0$ for ${\bf i} \notin \mathcal{C}.$ We note that the assumptions on the kernels $w_j$ guarantee that $|\tilde{\alpha}^j_{\bf i}| \le \tilde{\alpha}_* h$ for some constant $\tilde{\alpha}_*$ independent of $h.$ By linearity, if we use the same interpolation rule for $T_{\bf i'}^h$ and $S_{\bf i'}^h$ as for $L^h,$ we can define $T_{\bf i'}^h,S_{\bf i'}^h \in C(\Omega),$ with $L_{\bf i'}^h(\bx) = T_{\bf i'}^h(\bx) + S_{\bf i'}^h(\bx).$

We first consider $E= \sum_{\bf i'} |S_{\bf i'}^h(\bx)-S_{\bf i'}^h(\by)|$ and note from the definition of $S_{\bf i}^h$ that
$$E =\sum_{{\bf i'}} |\chi_{{\bf i'} \in \mathcal{C}^h(\bx)}S_{\bf i'}^h(\bx) -\chi_{{\bf i'} \in \mathcal{C}^h(\by)} S_{\bf i'}^h(\by)|.$$
For any point $\bx \in \Omega_0$ we observe that by construction there exist coefficients $q_{\bf i}(x)$ such that
$$\chi_{{\bf i'} \in \mathcal{C}^h(\bx)} S_{\bf i'}^h(\bx) = \sum_{\bf i} q_{\bf i}(\bx) \chi_{{\bf i'} \in \mathcal{C}^h(\bxi} S_{\bf i'}^h(\bxi)=\sum_{\bf i} q_{\bf i}(x) \chi_{{\bf i - \bf i'} \in \mathcal{C}} S_{\bf i'}^h(\bxi).$$
Moreover, for all $\bf i$, $q_{\bf i}(\bx)$ is a Lipschitz function of $\bx$ with Lipschitz constant bounded by $Q/h$ for some universal constant $Q.$ Finally, $\sum_{\bf i} |q_{\bf i}(\bx)| = 1,$ and for any $\bx \in \Omega,$ $q_{\bf i}(\bx)$ has at most 3 nonzero entries, thinking of it as a vector indexed by ${\bf i}$. Then,
$$E = \sum_{\bf i} |q_{\bf i}(\bx) - q_{\bf i}(\by)| \sum_{\bf i'} \chi_{{\bf i - \bf i'} \in \mathcal{C}}| S_{\bf i'}^h(x_{\bf i}^h)|.$$
We next observe that for ${\bf i-i'}\in\mathcal{C},$
$$ |S_{\bf i'}^h(\bx_{\bf i}^h)| \le \beta_* \sum_{j=1}^m|\tilde{\alpha}_{\bf i - \bf i'}^j(h)|\le m \beta_*\tilde{\alpha}_* h.$$
Thus,
$$E \le 6Q \tilde{\alpha}_* \beta_* m |\mathcal{C}|\, |\bx-\by|.$$

Turning to $T_{\bf i'}^h,$ we first consider the case in which $\bx$ and $\by$ are grid points $\bxl$ and $\bxlp,$ respectively. Next, we observe that from our assumptions,
$$|\nabla (\beta_j(\bx) w_j)| \le \frac{|\nabla \beta_j(\bx)|w_*}{|\bx|} + \frac{|\beta_j(\bx)|w_d}{|\bx|^2}\le \frac{\tilde{C}}{|\bx|^2}.$$
and, in particular,
$$|T_{\bf i'}^h(\bxl) - T_{\bf i'}^h(\bxlp)| \le {\tilde{C}}h |{\boldsymbol{\ell}  - \boldsymbol{\ell}'}| \sup_{t\in [0,1]}\frac{1}{|t \boldsymbol{\ell}+ (1-t)\boldsymbol{\ell}'-{\bf i'}|^2}.$$
Set $\boldsymbol{\ell}_*$ to be the nearest grid point to $(\boldsymbol{\ell}+\boldsymbol{\ell}')/2.$ Then, 
\begin{multline*} 
\sum_{|{\bf i'} -\bell_*|>h^{\eta-1}|\bell-\bell'|^\eta + 3 |\bell-\bell'|} |T_{\bf i'}^h(\bxl)-T_{\bf i'}^h(\bxlp)|\\ 
\le C' h { |\bell-\bell'|} \int_{{h^{\eta-1}|\bell- \bell'|^\eta \le r \le \frac{2\,{\rm diam}\,\Omega_0}{h}}} \frac{1}{r}\,{\rm d}r\le C'' {h|\bell-\bell'| 
(1+|\log{h|\bell-\bell'|}|)\; ,}
\end{multline*}
where the constants $C'$ and $C''$ are independent of $\bell,\bell',$ and $h.$ In the inner region, $|{\bf i'} -\bell_*|\le h^{\eta-1}|\bell-\bell'|^\eta + 3 |\bell-\bell'|$ we use the bounds on the $w_j$ directly, to obtain
\begin{align*}&\sum_{|{\bf i'} -\bell_*|\le h^{\eta-1}|\bell-\bell'|^\eta + 3 |\bell-\bell'|}|T_{\bf i'}^h(\bx)-T_{\bf i'}^h(\by)| \le \sum_{|{\bf i'} -\bell_*|\le h^{\eta-1}|\bell-\bell'|^\eta + 3 |\bell-\bell'|}|T_{\bf i'}^h(\bx)|+|T_{\bf i'}^h(\by)|\\
& \quad \le 2 C'''' h\int_0^{h^{\eta-1}|\bell- \bell'|^\eta + 3 |\bell - \bell'|} 1\,{\rm d}r\\
&\quad \le C''''' h^{\eta}|\bell-\bell'|^\eta = C''''' |\bxl - \bxlp|^\eta.
\end{align*}
Combining this with the previous estimate for the outer region gives
$$\sum_{\bf i'} |T_{\bf i'}^h(\bxl)-T_{\bf i'}^h(\bxlp)|\le C |\bxl-\bxlp|^\eta \, , $$
for some constant $C$ independent of $h,\bell,$ and $\bell'.$

We now turn to the case of arbitrary $\bx$ and $\by$. From the above estimate, the difference between the sum at neighboring grid points is bounded by $C h^\eta$, and hence, for $\bx$ and $\by$ in the same triangle,
$$ \sum_{\bf i'} |T_{\bf i'}^h(\bx)-T_{\bf i'}^h(\by)| \le C'|\bx-\by| h^{\eta-1}\le C'' |\bx -\by|^\eta.$$
The first inequality above gives that $T_{\bf i'}^h$, viewed as a vector-valued function of $\bx$ equipped with the $\ell_1$ norm, is Lipschitz, with Lipschitz constant bounded by $C' h^{\eta-1}.$ If $|\bx-\by| < 10 h$ then it follows that
$$ \sum_{\bf i'} |T_{\bf i'}^h(\bx)-T_{\bf i'}^h(\by)| \le C' |\bx-\by | h^{\eta-1} \le  C' |\bx-\by|^{1-\eta+\eta} h^{\eta-1} \le 10 C' |\bx-\by|^\eta. $$

Now suppose that $|\bx-\by| >10 h.$ Let $\bxl$ and $\bxlp$ be the closest grid points to $\bx$ and $\by$, respectively. Clearly, $|\bx-\by| > 1/2 ( |\bx-\bxl | +|\bxl-\bxlp | + |\bxlp - \by |). $
From the above inequalities it follows that
\begin{align*}
\sum_{\bf i'} |T_{\bf i'}^h(\bx)-T_{\bf i'}^h(\by)| &\le C'' (|\bx-\bxl|^\eta +|\bxlp-\by|^\eta) + C |\bxl-\bxlp|^\eta\\
&\le 3^{1-\eta}(C''+C)(|\bx-\bxl| + |\by-\bxlp| + |\bxl-\bxlp|)^\eta \\
 & \le 3^{1-\eta}2^{\eta}(C''+C) |\bx-\by|^\eta.
\end{align*}
where the second to last inequality follows from H\"{o}lder's inequality.

To conclude, {combining} the bounds on the differences of the $S^h$ and $T^h$ terms, we see that for all $|\bx-\by| \le 1,$
$$\sum_{\bf i'} |L_{\bf i'}^h(\bx) - L_{\bf i'}^h(\by)| \le C_* |\bx-\by|^{\eta} \, , $$
independent of $h.$ The constant $C_*$ does not depend on $\bx,\by,$ or $h,$ though it does depend on $\eta.$
\end{proof}